\documentclass[twoside,leqno,10pt, A4]{amsart}
\usepackage{amsfonts}
\usepackage{amsmath}
\usepackage{amscd}
\usepackage{amssymb}
\usepackage{amsthm}
\usepackage{amsrefs}
\usepackage{latexsym}
\usepackage{mathrsfs}
\usepackage{bbm}
\usepackage{enumerate}
\usepackage{graphicx}

\usepackage{amsfonts}
\usepackage{amsmath}
\usepackage{amscd}
\usepackage{amssymb}
\usepackage{amsthm}
\usepackage{amsrefs}
\usepackage{latexsym}
\usepackage{mathrsfs}
\usepackage{bbm}
\usepackage{amscd}
\usepackage{amssymb}
\usepackage{amsthm}
\usepackage{amsrefs}
\usepackage{latexsym}
\usepackage{mathrsfs}
\usepackage{bbm}
\usepackage{enumerate}
\usepackage{graphicx}
\usepackage{color}
\begin{document}

\newtheorem{theorem}[subsection]{Theorem}
\newtheorem{proposition}[subsection]{Proposition}
\newtheorem{lemma}[subsection]{Lemma}
\newtheorem{corollary}[subsection]{Corollary}
\newtheorem{conjecture}[subsection]{Conjecture}
\newtheorem{prop}[subsection]{Proposition}
\numberwithin{equation}{section}
\newcommand{\mr}{\ensuremath{\mathbb R}}
\newcommand{\mc}{\ensuremath{\mathbb C}}
\newcommand{\dif}{\mathrm{d}}
\newcommand{\intz}{\mathbb{Z}}
\newcommand{\ratq}{\mathbb{Q}}
\newcommand{\natn}{\mathbb{N}}
\newcommand{\comc}{\mathbb{C}}
\newcommand{\rear}{\mathbb{R}}
\newcommand{\prip}{\mathbb{P}}
\newcommand{\uph}{\mathbb{H}}
\newcommand{\fief}{\mathbb{F}}
\newcommand{\majorarc}{\mathfrak{M}}
\newcommand{\minorarc}{\mathfrak{m}}
\newcommand{\sings}{\mathfrak{S}}
\newcommand{\fA}{\ensuremath{\mathfrak A}}
\newcommand{\mn}{\ensuremath{\mathbb N}}
\newcommand{\mq}{\ensuremath{\mathbb Q}}
\newcommand{\half}{\tfrac{1}{2}}
\newcommand{\f}{f\times \chi}
\newcommand{\summ}{\mathop{{\sum}^{\star}}}
\newcommand{\chiq}{\chi \bmod q}
\newcommand{\chidb}{\chi \bmod db}
\newcommand{\chid}{\chi \bmod d}
\newcommand{\sym}{\text{sym}^2}
\newcommand{\hhalf}{\tfrac{1}{2}}
\newcommand{\sumstar}{\sideset{}{^*}\sum}
\newcommand{\sumprime}{\sideset{}{'}\sum}
\newcommand{\sumprimeprime}{\sideset{}{''}\sum}
\newcommand{\sumflat}{\sideset{}{^\flat}\sum}
\newcommand{\shortmod}{\ensuremath{\negthickspace \negthickspace \negthickspace \pmod}}
\newcommand{\V}{V\left(\frac{nm}{q^2}\right)}
\newcommand{\sumi}{\mathop{{\sum}^{\dagger}}}
\newcommand{\mz}{\ensuremath{\mathbb Z}}
\newcommand{\leg}[2]{\left(\frac{#1}{#2}\right)}
\newcommand{\muK}{\mu_{\omega}}
\newcommand{\thalf}{\tfrac12}
\newcommand{\lp}{\left(}
\newcommand{\rp}{\right)}
\newcommand{\Lam}{\Lambda_{[i]}}
\newcommand{\lam}{\lambda}
\def\L{\fracwithdelims}
\def\om{\omega}
\def\pbar{\overline{\psi}}
\def\phi{\varphi}
\def\phis{\phi^*}
\def\lam{\lambda}
\def\lbar{\overline{\lambda}}
\newcommand\Sum{\Cal S}
\def\Lam{\Lambda}
\newcommand{\sumtt}{\underset{(d,2)=1}{{\sum}^*}}
\newcommand{\sumt}{\underset{(d,2)=1}{\sum \nolimits^{*}} \widetilde w\left( \frac dX \right) }

\theoremstyle{plain}
\newtheorem{conj}{Conjecture}
\newtheorem{remark}[subsection]{Remark}

\providecommand{\re}{\mathop{\rm Re}}
\providecommand{\im}{\mathop{\rm Im}}
\def\cI{\mathcal{I}}
\def\cL{\mathcal{L}}
\def\E{\mathbb{E}}

\makeatletter
\def\widebreve{\mathpalette\wide@breve}
\def\wide@breve#1#2{\sbox\z@{$#1#2$}%
     \mathop{\vbox{\m@th\ialign{##\crcr
\kern0.08em\brevefill#1{0.8\wd\z@}\crcr\noalign{\nointerlineskip}%
                    $\hss#1#2\hss$\crcr}}}\limits}
\def\brevefill#1#2{$\m@th\sbox\tw@{$#1($}%
  \hss\resizebox{#2}{\wd\tw@}{\rotatebox[origin=c]{90}{\upshape(}}\hss$}
\makeatletter

\title[Lower bounds for negative moments of $\zeta'(\rho)$]{Lower bounds for negative moments of $\zeta'(\rho)$}

\author{Peng Gao and Liangyi Zhao}

\begin{abstract}
 We establish lower bounds for the discrete $2k$-th moment of the derivative of the Riemann zeta function at nontrivial zeros for all $k<0$ under the Riemann hypothesis (RH) and the assumption that all zeros of $\zeta(s)$ are simple.
\end{abstract}

\maketitle

\noindent {\bf Mathematics Subject Classification (2010)}: 11M06, 11M26 \newline

\noindent {\bf Keywords}: lower bounds, negative moments, nontrivial zeros, Riemann zeta function

\section{Introduction}
\label{sec 1}

  It is an important subject to study various types of moments of the Riemann zeta function $\zeta(s)$ on the critical line as they have many interesting applications. Take for instance,  the discrete $2k$-th moment $J_k(T)$ of the derivative of $\zeta(s)$ at its nontrivial zeros $\rho$, given by
\begin{align*}
J_k(T) =\sum_{0<\Im(\rho)\le T}|\zeta'(\rho)|^{2k}.
\end{align*}

  Note that $J_k(T)$ is well-defined for all real $k \geq 0$ with $J_0(T)=N(T)$, the number of zeros of $\zeta(s)$
in the rectangle with vertices $0, 1, 1 +iT$ and $iT$. In this paper, we assume that all zeros $\zeta(s)$ are simple which is widely believed to true. This allows one to extend the definition of $J_k(T)$ to hold for all real $k$. Then it is known (see \cite[Theorem 14.27]{ET} and \cite{Ng04}) that the behavior of $J_k(T)$ for $k<0$ is closely related to the size of the summatory function of the M\"obius function. \newline

  Independently,  S. M. Gonek \cite{Gonek1} and D. Hejhal \cite{Hejhal}  conjectured that for any real $k$,
\begin{align}
\label{Jksim}
J_k(T) \asymp T (\log T)^{(k+1)^2}.
\end{align}
  Via the random matrix theory, C. P. Hughes, J. P. Keating and N. O'Connell \cite{HKO} conjectured an asymptotic formula for $J_k(T)$, which also suggests that \eqref{Jksim} may not be valid for $k \leq -3/2$. This formula was recovered by H. M. Bui, S. M. Gonek and M. B. Milinovich \cite{BGM} using a hybrid Euler-Hadamard product. \newline

 Assuming the truth of the Riemann hypothesis (RH), S. M. Gonek \cite{Gonek} proved an asymptotic formula for $J_1(T)$ and N. Ng \cite{Ng} established \eqref{Jksim} for $k=2$.  Much progress has since been made  towards establishing \eqref{Jksim} for the case $k \geq 0$.  In fact, based on the work in \cites{MN1, Milinovich, Kirila, Gao2021-5}, we know that \eqref{Jksim} is valid for all real $k \geq 0$ on RH. \newline

 On the other hand, there are relatively fewer results for \eqref{Jksim} with $k<0$. Under RH and the assumption that all $\rho$'s are simple, S. M. Gonek \cite{Gonek1} obtained the sharp lower bound for $J_{-1}(T)$. In \cite{MN}, M. B. Milinovich and N. Ng further computed the implied constant to show that for any $\varepsilon>0$, 
\begin{align} \label{Jlowerbound}
J_{-1}(T) \geq (1-\varepsilon)\frac {3}{2\pi^3}T.
\end{align}
The bound in \eqref{Jlowerbound} is consistent with \eqref{Jksim} with the constant involved being half of that conjectured in \cites{Gonek1, HKO}.  Under the same assumptions, W. Heap, J. Li and J. Zhao \cite{HLZ} proved that the lower bounds implicit in the ``$\asymp$" of \eqref{Jksim} hold for all rational $k<0$. \newline

  It is the aim of this paper to further extend the lower bounds for $J_k(T)$ obtained in \cite{HLZ} to all real $k<0$. Our result is as follows.
\begin{theorem}
\label{thmlowerboundJ}
   Assuming RH and zeros of $\zeta(s)$ are simple. For large $T$ and any $k <0$, we have
\begin{align*}
   J_k(T) \gg_k T (\log T)^{(k+1)^2}.
\end{align*}
\end{theorem}

  We note that the approach in \cite{HLZ} follows the method developed by Z. Rudnick and K. Soundararajan in \cites{R&Sound} and also makes use of ideas of V. Chandee and X. Li \cite{C&L} when studying lower bounds for fractional moments of Dirichlet $L$-functions.  The approach in this paper uses a variant of the lower bounds principle of W. Heap and K. Soundararajan \cite{H&Sound} for proving lower bounds for general families of $L$-functions. This variant is motivated by a similar one used in \cite{HLZ}. In the proof of Theorem \ref{thmlowerboundJ}, we shall also adapt certain treatments of S. Kirila  \cite{Kirila} concerning sharp upper bounds for $J_k(T)$ for $k \geq 0$. These treatments orginate from a method of K. Soundararajan \cite{Sound01} and its refinement by A. J. Harper \cite{Harper}. \newline

  We also point out here that, as already mentioned above, the lower bounds for $J_k(T)$ given in Theorem \ref{thmlowerboundJ} are only expected to be sharp for $k \geq -3/2$.

\section{Preliminaries}
\label{sec 2}

 We now include some auxiliary results that are needed in the paper.  Recall that $N(T)$ denotes the number of zeros of $\zeta(s)$
in the rectangle with vertices $0, 1, 1 +iT$ and $iT$.  The Riemann–von Mangoldt formula asserts (see \cite[Chapter 15]{Da}) that
\begin{align}
\label{NT}
   N(T)=\frac {T}{2\pi}\log \frac {T}{2\pi e}+O(\log T).
\end{align}

   As we assume the truth of RH, we may write each non-trivial zero $\rho$ of $\zeta(s)$ as $\rho=\half+i\gamma$, where we write $\gamma \in \mr$ for the imaginary part of $\rho$.  We set $N(T, 2T)=N(2T)-N(T)$ so that \eqref{NT} implies
\begin{align}
\label{N2Tbound}
   N(T, 2T) \ll T \log T.
\end{align}

As usual, $\Lambda(n)$ stands for the von Mangoldt function and we extend the definition of $\Lambda$ to all real numbers $x$ by defining $\Lambda(x)=0$ for $x \notin \mz$. We then have the following uniform version of Landau’s formula \cite{Landau1912}, originally established by S. M. Gonek \cite{Gonek93}.
\begin{lemma}
\label{Lem-Landau}
	Assume RH. We have for $T$ large and any positive integers $a, b$,
\begin{align}
\label{sumgamma}
\begin{split}
	\sum_{T<\gamma\le 2T}(a/b)^{i\gamma}=
\begin{cases}
N(T, 2T), \quad a=b, \\
\displaystyle -\frac{T}{2\pi}\frac{\Lambda(a/b)}{\sqrt{a/b}}+O\big(\sqrt{ab}(\log T)^2\big), \quad a>b, \\
\displaystyle -\frac{T}{2\pi}\frac{\Lambda(b/a)}{\sqrt{b/a}}+O\big(\sqrt{ab}(\log T)^2\big), \quad b>a.
\end{cases}
\end{split}
\end{align}
\end{lemma}
 Note that the cases in which $a \neq b$ of Lemma \ref{Lem-Landau} are given in \cite[Lemma 5.1]{Kirila}, while the case with $a =b$ is trivial. \newline

   We reserve the letter $p$ for a prime number in this paper and we recall the following result from parts (d) and (b) in \cite[Theorem 2.7]{MVa} concerning sums over primes.
\begin{lemma}
\label{RS} Let $x \geq 2$. We have for some constant $b$,
\begin{align*}
\sum_{p\le x} \frac{1}{p} = \log \log x + b+ O\Big(\frac{1}{\log x}\Big).
\end{align*}
 Moreover, we have
\begin{align*}
\sum_{p\le x} \frac {\log p}{p} = \log x + O(1).
\end{align*}
\end{lemma}

  We also include the following mean value estimate similar to \cite[Lemma 3]{Sound01}.
\begin{lemma}
\label{lem:2.5}
Let $m$ be a natural number, we have for any $a(p) \geq 0$,
 \begin{align*}
   \sum_{T<\gamma\leq 2T}  \left|\sum_{\substack{p \leq y}}\frac{a(p) p^{i\gamma}}{p^{1/2}}\right|^{2m}
  \ll_{\varepsilon} & T(\log T) m!\left (\sum_{\substack{p \leq y}}\frac{a(p)}{p}\right)^m+(\log T)^2 \left( \sum_{\substack{p \leq y}}a(p) \right)^{2m}, \\
\sum_{T<\gamma\leq 2T}  \left|\sum_{\substack{p \leq y}}\frac{a(p) p^{2i\gamma}}{p}\right|^{2m}
  \ll_{\varepsilon} & T(\log T) m!\left (\sum_{\substack{p \leq y}}\frac{a(p)}{p^2}\right)^m+(\log T)^2 \left( \sum_{\substack{p \leq y}}a(p) \right)^{2m}.
 \end{align*}
\end{lemma}
\begin{proof}
  The proof is similar to that of \cite[Lemma 3]{Sound01}.  Hence, we only prove the first statement of the lemma. We write
\begin{align*}
  \left (\sum_{\substack{p \leq y}}\frac{a(p)p^{i\gamma}}{p^{1/2}}\right)^{m}=\sum_{f \leq y^m}\frac {a_{m,y}(f)f^{i\gamma}}{\sqrt{f}},
\end{align*}
  where $a_{m,y}(f) = 0$ unless $f$ is the product of $m$ (not necessarily distinct) primes not exceeding $y$. In that case, if we write the prime factorization of $f$ as $f = \prod^r_{j=1}p_j^{\alpha_j}$, then $a_{m,y}(f) = \binom{m}{\alpha_1, \ldots, \alpha_r}\prod^r_{j=1}a(p_j)^{\alpha_j} \geq 0$. \newline

  Now, as $a_{m,y}(f) \geq 0$, we apply Lemma \ref{Lem-Landau} to see that
\begin{align*}
\begin{split}
 \sum_{T<\gamma\leq 2T}  \left|\sum_{\substack{p \leq y}}\frac{a(p)p^{i\gamma}}{p^{1/2}}\right|^{2m}=& \sum_{f, g \leq y^m}\frac {a_{m,y}(f)\overline{a_{m,y}(g)}}{\sqrt{fg}}\sum_{T<\gamma\leq 2T} f^{i\gamma}g^{-i\gamma} \\
\leq &  N(T, 2T)\sum_{\substack{f \leq y^m}}\frac {a_{m,y}(f)^2}{f}+O((\log T)^2\sum_{f, g \leq y^m}a_{m,y}(f)\overline{a_{m,y}(g)}) \\
\leq &  N(T, 2T)\sum_{\substack{f \leq y^m}}\frac {a_{m,y}(f)^2}{f}+O((\log T)^2 (\sum_{\substack{p \leq y}}a(p))^{2m}).
\end{split}
\end{align*}
  We further estimate right-hand side expression above following exactly the treatments in \cite[Lemma 3]{Sound01} and utilizing \eqref{N2Tbound} to arrive at the desired result.
\end{proof}

\section{Proof of Theorem \ref{thmlowerboundJ}}
\label{sec 2'}

\subsection{The lower bound principle}

    We assume that $T$ is a large number and note here that throughout the proof, the implicit constants involved in various estimates in $\ll$ and $O$ notations depend on $k$ only and are uniform with respect to $\rho$. We recall the convention that the empty product is defined to be $1$. \newline

 We follow the ideas of A. J. Harper in \cite{Harper} and the notations of S. Kirila in \cite{Kirila} to define for a large number $M$ depending on $k$ only,
\begin{equation} \label{alphaJdef}
 \alpha_{0} = 0, \;\;\;\;\; \alpha_{j} = \frac{20^{j-1}}{(\log\log T)^{2}} \;\;\; \forall \; j \geq 1, \quad
\mathcal{J} = \mathcal{J}_{k,T} = 1 + \max\{j : \alpha_{j} \leq 10^{-M} \} .
\end{equation}

   We define $I_j=(T^{\alpha_{j-1}}, T^{\alpha_{j}}]$ for $1 \leq j \leq \mathcal{J}$.  With the above notations,  we can infer from Lemma \ref{RS} that for $1 \leq j \leq \mathcal{J}-1$ and $T$ large enough,
\begin{align}
\label{sumpj}
 \sum_{ p \in I_{j+1}} \frac{1}{p}
 = \log \alpha_{j+1} - \log \alpha_{j} + o(1) = \log 20 + o(1) \leq 10.
\end{align}

 We further define for any real number $\alpha$ and any $1\leq j \leq \mathcal{J}$,
\begin{align}
\label{defN}
 {\mathcal P}_j(s)=&  \sum_{ p \in I_j}  \frac{1}{p^s}, \quad {\mathcal N}_j(s, \alpha) = E_{e^2\alpha^{-3/4}_j} \Big (\alpha {\mathcal P}_j(s) \Big ), \quad  {\mathcal N}(s, \alpha)=  \prod^{\mathcal{J}}_{j=1} {\mathcal N}_j(s, \alpha),
\end{align}
  where for any real number $\ell$ and any $x \in \mr$,
\begin{align*}
  E_{\ell}(x) = \sum_{j=0}^{\lceil \ell \rceil} \frac {x^{j}}{j!}.
\end{align*}
Here $\lceil \ell \rceil = \min \{ m \in \intz : \ell \leq m \}$ is the celing of of $\ell$. \newline

 Now, applying a variant of the treatment  given by W. Heap, J. Li and J. Zhao in \cite[(3)]{HLZ}, which is based on the lower bounds principle of W. Heap and K. Soundararajan developed in \cite{H&Sound}, we have for $k<0$, from H\"older's inequality,
\begin{align}
\label{basiclowerbound}
\begin{split}
\sum_{0<\gamma\leq T}|\mathcal{N}(\rho, k)|^2
 \leq & \Big ( \sum_{0<\gamma\leq T}|\zeta'(\rho)|^{2} |\mathcal{N}(\rho, k)|^{2(1-1/k)}  \Big)^{-k/(1-k)}\Big ( \sum_{0<\gamma\leq T}  |\zeta'(\rho)|^{2k}\Big)^{1/(1-k)}.
\end{split}
\end{align}

By RH, $\overline{\rho}=1-\rho$.  So
\begin{align}
\label{sumoverrho}
\sum_{0<\gamma\leq T}|\mathcal{N}(\rho, k)|^2= \sum_{0<\gamma\leq T}\mathcal{N}(\rho, k)\mathcal{N}(\overline{\rho}, k)=\sum_{0<\gamma\leq T}\mathcal{N}(\rho, k)\mathcal{N}(1-\rho, k).
\end{align}

Hence from \eqref{basiclowerbound} and \eqref{sumoverrho}, in order to establish Theorem \ref{thmlowerboundJ}, it suffices to prove the following propositions.
\begin{proposition}
\label{Prop4}
  With notations as above, we have for $k < 0$,
\begin{align*}
 \sum_{0<\gamma\leq T}\mathcal{N}(\rho, k)\mathcal{N}(1-\rho, k) \gg T (\log T)^{ k^2+1}.
\end{align*}
\end{proposition}

\begin{proposition}
\label{Prop6}
  With notations as above, we have for $k <0$,
\begin{align*}
\sum_{0<\gamma\leq T}|\zeta'(\rho)|^{2} |\mathcal{N}(\rho, k)|^{2(1-1/k)}  \ll T ( \log T  )^{k^2+3}.
\end{align*}
\end{proposition}

The remainder of the paper is devoted to the proofs of the above propositions.

\section{Proof of Proposition \ref{Prop4}}
\label{sec 4.8}

   We write $\omega(n)$ for the number of distinct prime factors of $n$ and $\Omega(n)$ for the number of prime powers dividing $n$.  We also
denote $g(n)$ for the multiplicative function given on prime powers by $g(p^{r}) = 1/r!$ and define functions $b_j(n), 1 \leq j \leq {\mathcal{J}}$  such that $b_j(n)=0$ or $1$ and that $b_j(n)=1$ only if $\Omega(n) \leq \ell_j:=\lceil e^2\alpha^{-3/4}_j \rceil$ and all prime divisors of $n$ lies in the interval $I_j$. We can then rewrite $\mathcal{N}_j(s,\alpha)$ defined \eqref{defN}, using the above notations, as
\begin{align}
\label{Nj}
 {\mathcal N}_j(s,\alpha) = \sum_{n_j}  \frac{\alpha^{\Omega(n_j)}}{g(n_j)}  b_j(n_j) \frac 1{n^s_j}, \quad 1\le j \le {\mathcal{J}}.
\end{align}

  Note that each ${\mathcal N}_j(s,\alpha)$ is a short Dirichlet polynomial of length not exceeding $T^{\alpha_{j}\lceil e^2\alpha^{-3/4}_j \rceil}$. We also write for simplicity that
\begin{align*}
 {\mathcal N}(s, \alpha)= \sum_{n} \frac{a_{\alpha}(n)}{n^s}.
\end{align*}
   Then it follows from \eqref{defN} and \eqref{Nj} that $a_{\alpha}(n) \neq 0$ only when $n=\prod_{1\leq j \leq \mathcal{J}}n_j$ and in that case
\begin{align} \label{aalpha}
  a_{\alpha}(n)= \prod_{n_j}  \frac{\alpha^{\Omega(n_j)}}{g(n_j)}  b_j(n_j).
\end{align}

  Upon taking $T$ large enough, we then deduce that ${\mathcal N}(s, \alpha)$ is a Dirichlet polynomial not longer than
\begin{align*}
 T^{\sum^{\mathcal{J}}_{j=1} \alpha_{j}\lceil e^2\alpha^{-3/4}_j \rceil} \leq T^{40e^210^{-M/4}}.
\end{align*}

   Moreover, applying the estimation that $\big|\alpha^{i}/i!\big | \leq  e^{|\alpha|}$ for any integer $i \geq 0$,
we get that
\begin{align}
\label{anbound}
  |a_{\alpha}(n)| \leq e^{|\alpha|\omega(n)}, \quad a_{\alpha}(n) = 0, \; \text{if} \; n > T^{40 e^210^{-M/4}}.
\end{align}

   Note further the following estimation (see \cite[Theorem 2.10]{MVa}),
\begin{align}
\label{omegabound}
  \omega(n) \leq \frac {\log n}{\log \log n} \left( 1+O\left( \frac {1}{\log \log n} \right) \right), \quad n \geq 3.
\end{align}

  It follows from \eqref{anbound} and \eqref{omegabound} that, upon taking $M$ large enough, we have
\begin{align*}
 {\mathcal N}(s, k)= \sum_{n \leq T^{\theta}} \frac{a_{k}(n)}{n^s},
\end{align*}
  where $\theta <1$ and $|a_{k}(n)| \leq T^{\varepsilon}$. \newline

  We then apply \cite[Proposition 3.1(iii)]{Ng08} to arrive at
\begin{align}
\label{sumNsquare}
 \sum_{0<\gamma\leq T}\mathcal{N}(\rho, k)\mathcal{N}(1-\rho, k) = N(T) \sum_{n}\frac {a^2_{k}(n)}{n}-
\frac {T}{\pi}\sum_{n}\frac {(\Lambda*a_k)(n)\cdot a_{k}(n)}{n}+o(T),
\end{align}
 where  $f*g$ is the Dirichlet convolution of two arithmetic functions $f(k), g(k)$ defined by
\begin{align*}
  (f*g)(k)=\sum_{mn=k}f(m)g(n).
\end{align*}

  We evaluate the first sum in \eqref{sumNsquare} using \eqref{aalpha} to obtain that
\begin{align*}
\begin{split}
  \sum_{n}\frac {a^2_{k}(n)}{n}=& \prod^{\mathcal{J}}_{j=1}\Big ( \sum_{n_j} \frac{1}{n_j} \frac{k^{2\Omega(n_j)}}{g^2(n_j)}
b_j(n_j) \Big ).
\end{split}
\end{align*}

 Note that the factor $b_j(n_j)$ restricts $n_j$ to having all prime factors in $I_j$ such that $\Omega(n_j) \leq \ell_j$. If we remove this restrictions on $\Omega$, then the sum over $n_j$ becomes
\begin{align*}
\begin{split}
\sum_{n_j} \frac{1}{n_j} \frac{k^{2\Omega(n_j)}}{g^2(n_j)}=\prod_{\substack{p\in I_j }}\left(1+ \frac {k^2}p+O\left( \frac 1{p^2} \right) \right):=P_j.
\end{split}
\end{align*}

  Using Rankin's trick by noticing that $2^{\Omega(n_j)-\ell_j}\ge 1$ if $\Omega(n_j) > \ell_j$,  we see by the definition of $\ell_j$ and \eqref{sumpj} that the error introduced this way does not exceed
\begin{align*}
\begin{split}
  \sum_{n_j} \frac{1}{n_j} \frac{k^{2\Omega(n_j)}}{g^2(n_j)}2^{\Omega(n_j)-\ell_j}  \le & 2^{-\ell_j} \prod_{\substack{p\in I_j }} \Big( \sum_{i=0}^{\infty} \frac{1}{p^i} \frac{(2k^2)^{i}}{(i!)^2}\Big) \leq  2^{-\ell_j/2}P_j.
\end{split}
\end{align*}

  We conclude from the above discussions, \eqref{NT} and Lemma \ref{RS} that, by taking $T$ large enough,
\begin{align}
\label{sumasquare}
 N(T) \sum_{n}\frac {a^2_{k}(n)}{n} \gg   T \log T\prod^{\mathcal{J}}_{j=1}\Big (1-2^{-\ell_j/2}\Big )\prod^{\mathcal{J}}_{j=1}P_j
\gg T(\log T)^{k^2+1}.
\end{align}

  Next, we estimate
\begin{align*}
 \frac {T}{\pi}\sum_{n}\frac {(\Lambda*a_k)(n)\cdot a_{k}(n)}{n}=\frac{T}{\pi }\sum_{n}\frac{\Lambda(n)}{n}\sum_m\frac {a_{k}(m) a_k(mn)}{m} \leq \frac{T}{\pi }\sum_{n}\frac{\Lambda(n)}{n}\Big |\sum_m\frac {a_{k}(m) a_k(mn)}{m}\Big |.
\end{align*}

  Note that by \eqref{anbound}, we may assume that $n \leq T^{40 e^210^{-M/4}}$.  We fix an integer $n$ and write it as $n=\prod n_j$ with $p|n_j$ only if $p \in I_j$. Then we have
\begin{align*}
 \Big |\sum_m\frac {a_{k}(m) a_k(mn)}{m}\Big |=\Big |\prod^{\mathcal{J}}_{j=1}\Big ( \sum_{m_j} \frac{1}{m_j} \frac{k^{2\Omega(m_j)+\Omega(n_j)}}{g(m_j)g(m_jn_j)} b_j(m_jn_j) \Big ) \Big | \leq
\prod^{\mathcal{J}}_{j=1} \Big | \sum_{m_j} \frac{1}{m_j} \frac{k^{2\Omega(m_j)+\Omega(n_j)}}{g(m_j)g(m_jn_j)}
b_j(m_jn_j) \Big |.
\end{align*}
  Observe that $g(mn) \geq g(m)g(n)$ so that
\begin{align*}
\prod^{\mathcal{J}}_{j=1}\Big | \sum_{m_j} \frac{1}{m_j} \frac{k^{2\Omega(m_j)+\Omega(n_j)}}{g(m_j)g(m_jn_j)}
b_j(m_jn_j) \Big |
 \leq & \prod^{\mathcal{J}}_{j=1} |k|^{\Omega(n_j)}\Big ( \sum_{m_j} \frac{1}{m_j} \frac{k^{2\Omega(m_j)}}{g(m_j)g(m_jn_j)}
b_j(m_jn_j) \Big ) \\
\leq & \prod^{\mathcal{J}}_{j=1} \frac {|k|^{\Omega(n_j)}}{g(n_j)} \Big ( \sum_{m_j} \frac{1}{m_j} \frac{k^{2\Omega(m_j)}}{g^2(m_j)}
 \Big ) = \frac {|k|^{\Omega(n)}}{g(n)}\prod^{\mathcal{J}}_{j=1}  \Big ( \sum_{m_j} \frac{1}{m_j} \frac{k^{2\Omega(m_j)}}{g^2(m_j)}
 \Big ).
\end{align*}

  We evaluate the last expression above using Lemma \ref{RS} to see that we have
\begin{align*}
 \frac {T}{\pi}\sum_{n}\frac {(\Lambda*a_k)(n)\cdot a_{k}(n)}{n} \ll T(\log T)^{k^2}\sum_{n}\frac{\Lambda(n)|k|^{\Omega(n)}}{ng(n)}
\ll T(\log T)^{k^2} \sum_{p \leq T^{40 e^210^{-M/4}}}\frac{\log p}{p} \ll 10^{-M/4}T(\log T)^{k^2+1}.
\end{align*}

   We then take $M$ large enough to deduce from the above, \eqref{sumNsquare} and \eqref{sumasquare} that
\begin{align*}
  \sum_{0<\gamma\leq T}\mathcal{N}(\rho, k)\mathcal{N}(1-\rho, k)\gg & T(\log T)^{k^2+1}.
\end{align*}

   This completes the proof of Proposition \ref{Prop4}.

\section{Proof of Proposition \ref{Prop6}}
\label{sec: proof of Prop 6}

   Upon dividing the range of $\gamma$ into dyadic blocks and replacing $T$ by $2T$, we see that it suffices to show for large $T$,
\begin{align*}
\sum_{T<\gamma\leq 2T}|\zeta'(\rho)|^{2} |\mathcal{N}(\rho, k)|^{2(1-1/k)}   \ll T ( \log T  )^{k^2+3}.
\end{align*}

  We now deduce from \cite[(4.1)]{Kirila} the following upper bound for $\log |\zeta'(\rho)|$.  For any $\gamma$ with $T< \gamma\le 2T$ and any $1 \leq j \leq \mathcal{J}$,
\begin{align*}
\begin{split}
 & \log |\zeta'(\rho)| \ll \Re \sum^{j}_{l=1}\sum_{n\in I_l}\frac{n^{-i\gamma}}{\sqrt{n}}\frac{\Lambda_{\cL}(n)}{n^{1/(\alpha_j\log T)}\log n}\frac{\log (T^{\alpha_j}/n)}{\log T^{\alpha_j}}+\log \log T+ \alpha^{-1}_j+O(1),
\end{split}
\end{align*}
  where we define $\cL=\log T$ and
\begin{align*}
  \Lambda_{\cL}(n) =
\begin{cases}
  \Lambda(n),  \quad \text{ if } n=p \text{ or if } n=p^2 \text{ and } n \leq \cL, \\
   0,  \quad \text{otherwise}.
\end{cases}
\end{align*}

    Note that we have by Lemma \ref{RS} and the definition in \eqref{alphaJdef},
\begin{align}
\label{Jbound}
 \mathcal{J} \leq \log\log\log T  , \; \alpha_{1} = \frac{1}{(\log\log T)^{2}} , \; \mbox{and} \; \sum_{p  \leq T^{1/(\log\log T)^{2}}} \frac{1}{p} \leq \log\log T=\alpha^{-1/2}_1 .
\end{align}

   It follows that $1/\alpha_j \leq 1/\alpha_1 \leq (\log\log T)^{2}$ and this, together with Lemma \ref{RS} and the mean value theorem in differential calculus, implies that for $j \geq 1$,
\begin{align*}
\begin{split}
 \sum_{\substack{  p \leq \log T }} \frac{1}{p^{1+2/(\alpha_j\log T)}}  \frac{\log p}{\log T^{\alpha_j}}  \ll 1, \quad
  \sum_{\substack{  p \leq \log T }} \Big (\frac{1}{p^{1+2/(\alpha_j\log T)}}- \frac{1}{p} \Big ) & \ll \frac 1{\alpha_j\log T}\sum_{\substack{  p \leq \log T }}\frac{\log p}{p} \ll 1.
\end{split}
 \end{align*}

   We then conclude that
\begin{align}
\label{basicest}
\begin{split}
 & \log |\zeta'(\rho)| \ll \Re \sum^{j}_{l=1} {\mathcal M}_{l,j}(\gamma)+\Re \sum_{0 \leq m \leq \frac{\log\log T}{\log 2}} P_{m}(\gamma)+\log \log T+ \alpha^{-1}_j+O(1),
\end{split}
 \end{align}
  where
\[ {\mathcal M}_{l,j}(\gamma) = \sum_{p\in I_l}\frac{p^{-i\gamma}}{\sqrt{p}}\frac{1}{p^{1/(\alpha_j\log T)}}\frac{\log (T^{\alpha_j}/p)}{\log T^{\alpha_j}}, \quad 1\leq l \leq j \leq \mathcal{J} , \]
and
\[ P_{m}(\gamma )=  \sum_{2^{m} < p \leq 2^{m+1}} \frac{p^{-2i\gamma}}{2p} , \quad  0 \leq m \leq \frac{\log\log T}{\log 2}. \]

  We also define the following sets:
\begin{align*}
  \mathcal{S}(0) =& \{ T< \gamma \leq 2T : |{\mathcal M}_{1,l}(\gamma)| > \alpha_{1}^{-3/4} \; \text{ for some } 1 \leq l \leq \mathcal{J} \} ,   \\
 \mathcal{S}(j) =& \{ T< \gamma \leq 2T : |{\mathcal M}_{m,l}(\gamma)| \leq
 \alpha_{m}^{-3/4} \; \mbox{for all} \; 1 \leq m \leq j, \;  m \leq l \leq \mathcal{J}, \\
 & \;\;\;\;\; \text{but }  |{\mathcal M}_{j+1,l}(\gamma)| > \alpha_{j+1}^{-3/4} \; \text{ for some } j+1 \leq l \leq \mathcal{J} \} ,  \quad  1\leq j \leq \mathcal{J}, \\
 \mathcal{S}(\mathcal{J}) =& \{T< \gamma \leq 2T : |{\mathcal M}_{m,
\mathcal{J}}(\gamma)| \leq \alpha_{m}^{-3/4} \; \mbox{for all} \; 1 \leq m \leq \mathcal{J}\}, \\
\mathcal{P}(m) =&  \left\{ T< \gamma \leq 2T :  |P_{m}(\gamma)| > 2^{-m/10} , \; \text{but} \; |P_{n}(\gamma)| \leq 2^{-n/10} \; \mbox{for all} \; m+1 \leq n \leq \frac{\log\log T}{\log 2} \right\}.
\end{align*}

  Note that we have
$|P_{n}(\gamma)| \leq 2^{-n/10}$ for all $n$ if $\gamma \not \in \mathcal{P}(m)$ for any $m$, which implies that
\[ \sum_{\substack{  p \leq \log T }} \frac{p^{-2i\gamma}}{2p} = O(1). \]
 As the treatment for the case $\gamma \not \in \mathcal{P}(m)$ for any $m$ is easier compared to the other cases, we may assume that $\gamma \in \mathcal{P}(m)$ for some $m$.   Note that
\begin{align*}
\text{meas}(\mathcal{P}(m)) \leq  \sum_{\substack{T< \gamma \leq 2T }}
\Big (2^{m/10} | P_m(\gamma)| \Big )^{2\lceil 2^{m/2}\rceil } ,
\end{align*}
where here and after $\mathrm{meas}(X)$ denotes the cardinality of the set $X$. \newline

    We now apply Lemmas~\ref{RS} and \ref{lem:2.5} to bound the last expression above.  This gives that for $m \geq 10$,
\begin{align*}
\begin{split}
  \text{meas}(\mathcal{P}(m)) \ll T(\log T) (\lceil 2^{m/2}\rceil !) (2^{m/5})^{\lceil 2^{m/2}\rceil }\Big (\sum_{2^m < p } \frac {1}{4p^2}\Big )^{\lceil 2^{m/2}\rceil }+(\log T)^2 \Big( \sum_{ p \leq \log T} 1 \Big)^{2\lceil 2^{m/2}\rceil } \ll T  2^{-2^{m/2}},
\end{split}
\end{align*}
   where the last estimation above follows from Stirling's formula (see \cite[(5.112)]{iwakow}), which asserts that
\begin{align}
\label{Stirling}
\begin{split}
 \Big( \frac{m}{e} \Big)^m \leq m! \ll \sqrt{m} \Big( \frac {m }{e} \Big)^{m}.
\end{split}
\end{align}

Now H\"older's inequality gives that if $2^{m} \geq (\log\log T)^{3}$,
\begin{align*}
  \sum_{(\log\log T)^3 \leq 2^m \leq \log T} & \sum_{\gamma \in  \mathcal{P}(m)}|\zeta'(\rho)|^{2} |\mathcal{N}(\rho, k)|^{2(1-1/k)} \\
\leq & \sum_{(\log\log T)^3 \leq 2^m \leq \log T}  \Big ( \text{meas}(\mathcal{P}(m)) \Big )^{1/4} \Big (
\sum_{T<\gamma\leq 2T}|\zeta'(\rho)|^{8}  \Big )^{1/4} \Big ( \sum_{T<\gamma\leq 2T} |\mathcal{N}(\rho, k)|^{4(1-1/k)}  \Big )^{1/2}.
\end{align*}

  Similar to the proof of \cite[Proposition 3.5]{Gao2021-5}, we have that
\begin{align}
\label{N2k2bound}
&  \sum_{T<\gamma\leq 2T}  |\mathcal{N}(\rho, k)|^{4(1-1/k)}  \ll T( \log T  )^{O(1)}.
\end{align}

  Also, note that by \cite[Theorem 1.1]{Kirila}, we have under RH, for any real $k>0$
\begin{align*}
J_k(T) \ll_{k} T (\log T)^{(k+1)^2}.
\end{align*}
Putting $k=4$, we have
\begin{align} \label{L8bound}
\sum_{T<\gamma\leq 2T}|\zeta'(\rho)|^{8}  \ll T(\log T)^{25}.
\end{align}

  It thus follows that, as $2^m \geq (\log \log T)^3$,
\begin{align*}
 \sum_{(\log\log T)^3 \leq 2^m \leq \log T} &\sum_{\gamma \in  \mathcal{P}(m)}|\zeta'(\rho)|^{2} |\mathcal{N}(\rho, k)|^{2(1-1/k)} \\
 \ll & \sum_{(\log\log T)^3 \leq 2^m \leq \log T} T \exp\left( -(\log 2)(\log\log T)^{3/2}/4 \right)(\log T)^{O(1)}
 \ll T(\log T)^{k^2}(\log \log T)^{-1}.
\end{align*}

Hence we may also assume that $0 \leq m \leq (3/\log 2)\log\log\log T$. We further note that,
\begin{align*}
\begin{split}
\text{meas}(\mathcal{S}(0)) \leq & \sum_{T<\gamma\leq 2T}  \sum^{\mathcal{J}}_{l=1}
\Big ( \alpha^{3/4}_{1}{|\mathcal
M}_{1, l}(\chi)| \Big)^{2\lceil 1/(10\alpha_{1})\rceil }.
\end{split}
\end{align*}
Now Lemma \ref{lem:2.5} applied to the last expression above yeilds
\begin{align}
\label{Smest}
\begin{split}
 \text{meas}(\mathcal{S}(0)) \ll &  \mathcal{J}T(\log T) (\lceil 1/(10\alpha_{1})\rceil !) (\alpha^{3/4}_{1})^{2 \lceil 1/(10\alpha_{1})\rceil}\Big (\sum_{p  \leq T^{\alpha_1}} \frac{1}{p}\Big )^{ \lceil 1/(10\alpha_{1})\rceil} +\mathcal{J}(\log T)^2(\sum_{p  \leq T^{\alpha_1}}1)^{2\lceil 1/(10\alpha_{1})\rceil} \\
\ll & \mathcal{J} T(\log T) \sqrt{\lceil 1/(10\alpha_{1})\rceil }\big(\frac {\lceil 1/(10\alpha_{1})\rceil }{e}\big)^{ \lceil 1/(10\alpha_{1})\rceil} (\alpha^{3/4}_{1})^{2 \lceil 1/(10\alpha_{1})\rceil}\Big (\sum_{p  \leq T^{\alpha_1}} \frac{1}{p}\Big )^{ \lceil 1/(10\alpha_{1})\rceil},
\end{split}
\end{align}
 where the last estimation above follows from Lemma \ref{RS} and \eqref{Stirling}. \newline

   We apply the bounds in \eqref{Jbound} to the last expression in \eqref{Smest}, getting
\begin{align} \label{meass0est}
\text{meas}(\mathcal{S}(0)) \ll &
\mathcal{J}T(\log T)  \sqrt{\lceil 1/(10\alpha_{1})\rceil } e^{-1/(10\alpha_{1})}\ll T e^{-(\log\log T)^{2}/20}  .
\end{align}

Now H\"older's inequality renders
\begin{align}
\label{LS0bound}
\begin{split}
& \sum_{\gamma \in S(0)}|\zeta'(\rho)|^{2} |\mathcal{N}(\rho, k)|^{2(1-1/k)}
\leq  \Big ( \text{meas}(\mathcal{S}(0)) \Big )^{1/4} \Big (
\sum_{T<\gamma\leq 2T}|\zeta'(\rho)|^{8}  \Big )^{1/4} \Big ( \sum_{T<\gamma\leq 2T} |\mathcal{N}(\rho, k)|^{4(1-1/k)}   \Big )^{1/2}.
\end{split}
\end{align}

  We use the bounds \eqref{N2k2bound}, \eqref{L8bound} and \eqref{meass0est} in \eqref{LS0bound} to conclude that
\begin{align*}
\sum_{\gamma \in \mathcal{S}(0)}|\zeta'(\rho)|^{2} |\mathcal{N}(\rho, k)|^{2(1-1/k)}    \ll T(\log T)^{k^2+3}.
\end{align*}

  Similarly, we define
\begin{align*}
  \mathcal{T} =& \{ T< \gamma \leq 2T : |k{\mathcal P}_{1}(\rho)| \leq  \frac {\alpha_{1}^{-3/4}}{1-1/k} \}.
\end{align*}
  We further write $\mathcal{T}^c$ for the complementary of $\mathcal{T}$ in the set $\{ T< \gamma \leq 2T \}$. Then similar to our approach above, we have
\begin{align*}
\text{meas}(\mathcal{T}^c) \ll & T e^{-(\log\log T)^{2}/20}.
\end{align*}
  Our arguments above allow us to deduce
\begin{align*}
\sum_{\gamma \in \mathcal{T}^c}|\zeta'(\rho)|^{2} |\mathcal{N}(\rho, k)|^{2(1-1/k)}    \ll T(\log T)^{k^2+3}.
\end{align*}

  Thus we may further assume that $j \geq 1$ and $\gamma \in \mathcal{T}$. Note that
\begin{align*}
 & \left\{ \gamma \in \mathcal{T}, \gamma \in \mathcal{P}(m), \gamma \in \mathcal{S}(j), 0 \leq m \leq (3/\log 2)\log\log\log T, 1 \leq j \leq \mathcal{J} \right \} =  \bigcup_{m=0}^{(3/\log 2)\log\log\log T}\bigcup_{j=1}^{ \mathcal{J}} \Big (\mathcal{S}(j)\bigcap \mathcal{P}(m) \bigcap \mathcal{T}\Big ).
\end{align*}
 Thus, it suffices to show that
\begin{align}
\label{sumovermj}
  \sum_{m=0}^{(3/\log 2)\log\log\log T}\sum_{j=0}^{\mathcal{J}}\sum_{\gamma \in \mathcal{S}(j)\bigcap \mathcal{P}(m)\bigcap \mathcal{T}} |\zeta'(\rho)|^{2} |\mathcal{N}(\rho, k)|^{2(1-1/k)}
   \ll T(\log T)^{k^2+3}.
\end{align}

  Now,  we consider the sum of $|\zeta'(\rho)|^{2} |\mathcal{N}(\rho, k)|^{2(1-1/k)}$ over $\mathcal{S}(j)\bigcap \mathcal{P}(m) \bigcap \mathcal{T}$ by fixing an $m$ such that $0 \leq m \leq (3/\log 2)\log\log\log T$ and fixing a $j$ with $1 \leq j \leq \mathcal{J}$.  We then deduce from \eqref{basicest} that
\begin{align}
\label{zetaNbounds}
\begin{split}
   |\zeta'(\rho)|^{2} & |\mathcal{N}(\rho, k)|^{2(1-1/k)} \\
\ll & (\log T)^2\exp \left(\frac {2}{\alpha_j} \right) \exp \Big (
 2\Re\sum^j_{l=1}{\mathcal M}_{l,j}(\gamma)+2 \Re\sum^{\log \log T /2}_{m=0}P_m(\gamma) \Big )|\mathcal{N}(\rho, k)|^{2(1-1/k)} \\
\ll & (\log T)^2\exp \left(\frac {2}{\alpha_j} \right) \exp \Big (
 2\Re\sum^j_{l=1}{\mathcal M}_{l,j}(\gamma)+2 \Re\sum^{\log \log T /2}_{m=0}P_m(\gamma) \Big )|\mathcal{N}_1(\rho, k)|^{2(1-1/k)}\prod^{\mathcal J}_{l=2}|\mathcal{N}_l(\rho, k)|^{2(1-1/k)} .
\end{split}
\end{align}

  Now, for any $z \in \mc$, by the Taylor formula with integral remainder, we have
\begin{align*}
\begin{split}
 \Big|e^z-\sum^{d-1}_{j=0}\frac {z^j}{j!}\Big| =& \Big|\frac 1{(d-1)!}\int\limits^z_0e^t(z-t)^{d-1} \dif t\Big|=\Big|\frac {z^d}{(d-1)!}\int\limits^1_0 e^{zs}(1-s)^{d-1} \dif s\Big| \\
\leq & \frac {|z|^d}{d!}\max (1, e^{\Re z}) \leq \frac {|z|^d}{d!}e^z \max (e^{-z}, e^{\Re z-z}) \leq \frac {|z|^d}{d!}e^z e^{|z|}.
\end{split}
 \end{align*}

The above computation implies that
\begin{align}
\label{ezrelation}
\begin{split}
  \sum^{d-1}_{j=0}\frac {z^j}{j!} =e^z \left( 1+O \left( \frac {|z|^d}{d!}e^{|z|} \right) \right).
\end{split}
 \end{align}

   As $\gamma \in \mathcal{T}$, we apply the above formula that, by setting $z= k{\mathcal P}_{1}(\rho)$, $d=[e^2\alpha^{-3/4}_1]$ \eqref{ezrelation}.  In a manner similar to the bound after \cite[(5.2)]{Kirila} and applying \eqref{Stirling}, we get
\begin{align}
\label{N1bound}
\begin{split}
  |\mathcal{N}_1(\rho, k)|^2 = \exp\Big ( 2k \Re {\mathcal P}_{1}(\half+i\gamma)\Big ) (1+O(e^{-\alpha^{-3/4}_1})).
\end{split}
 \end{align}

Now applying \eqref{N1bound} in \eqref{zetaNbounds} yields
\begin{align*}
\begin{split}
 |\zeta'(\rho)|^{2} & |\mathcal{N}(\rho, k)|^{2(1-1/k)} \\
\ll & (\log T)^2\exp \left(\frac {2}{\alpha_j} \right) \exp \Big (
 2\Re {\mathcal M}_{1,j}(\gamma)+ 2(k-1)\Re {\mathcal P}_{1}(\half+i\gamma)+2\Re\sum^j_{l=2}{\mathcal M}_{l,j}(\gamma)+2 \Re\sum^{\log \log T /2}_{m=0}P_m(\gamma) \Big ) \\
& \hspace*{2cm} \times \prod^{\mathcal J}_{l=2}|\mathcal{N}_l(\rho, k)|^{2(1-1/k)} .
\end{split}
\end{align*}

   We want to now separate the sums over $p \leq 2^{m+1}$ on the right-hand side of the above expression from those over $p >2^{m+1}$. To do so, we note that if $\gamma \in \mathcal{P}(m)$, then
\begin{align*}
\begin{split}
 & \Big | 2\Re \sum_{  p  \leq 2^{m+1}}  \frac{p^{-i\gamma}}{\sqrt{p}}\frac{1}{p^{1/(\alpha_j\log T)}}\frac{\log (T^{\alpha_j}/p)}{\log T^{\alpha_j}}+2(k-1)\Re \sum_{  p  \leq 2^{m+1}} \frac{p^{-i\gamma}}{\sqrt{p}}+
  2\Re \sum_{p  \leq \log T } \frac{p^{-i\gamma}}{2p} \Big |
   \\
 \leq &  2\Big |\sum_{  p  \leq 2^{m+1}}  \frac{p^{-i\gamma}}{\sqrt{p}}\frac{1}{p^{1/(\alpha_j\log T)}}\frac{\log (T^{\alpha_j}/p)}{\log T^{\alpha_j}}\Big |+2(1-k)\Big |\sum_{  p  \leq 2^{m+1}} \frac{p^{-i\gamma}}{\sqrt{p}}\Big |+
  \Big |\sum_{p  \leq 2^{m+1}} \frac{p^{-i\gamma}}{2p}\Big |
  +O(1) \leq (4-2k)2^{m/2+3}+O(1),
\end{split}
\end{align*}
  where we note by Lemma \ref{RS} that (using $\log (1+x) \leq x$ and $x \leq 2^x$ (this last inequality is valid for $x \geq 1$ and $x=1/2$))
\begin{align*}
\begin{split}
  \Big |\sum_{p  \leq 2^{m+1}} \frac{p^{-i\gamma}}{2p}\Big | \leq \sum_{p  \leq 2^{m+1}} \frac{1}{2p}
   =\frac 12\log (m+1)+O(1) \leq \frac m2+O(1) \leq 2^{m/2}+O(1).
\end{split}
 \end{align*}

 We deduce from the above that
\begin{align}
\label{LboundinSP0}
\begin{split}
  &  \sum_{\gamma \in \mathcal{S}(j)\bigcap \mathcal{P}(m)\bigcap \mathcal{T}} |\zeta'(\rho)|^{2} |\mathcal{N}(\rho, k)|^{2(1-1/k)} \\
   \ll & (\log T)^2e^{(2-k)2^{m/2+4}} \exp \left(\frac {2}{\alpha_j} \right) \sum_{\chi \in \mathcal{S}(j)\bigcap \mathcal{P}(m)}
\exp \Big ( 2 \Re{\mathcal M}'_{1,j}(\gamma)+2\Re \sum^j_{l=2}{\mathcal M}_{l,j}(\gamma)\Big )\prod^{\mathcal J}_{l=2}|\mathcal{N}_l(\rho, k)|^{2(1-1/k)} \\
  \\
\ll &  (\log T)^2e^{(2-k)2^{m/2+4}} \exp \left(\frac {2}{\alpha_j} \right) \\
& \times \sum_{\gamma \in \mathcal{S}(j)} \Big (2^{m/10}|P_m(\gamma)| \Big )^{2\lceil 2^{m/2}\rceil }
\exp \Big ( 2 \Re{\mathcal M}'_{1,j}(\gamma)+2\Re \sum^j_{l=2}{\mathcal M}_{l,j}(\gamma)\Big )\prod^{\mathcal J}_{l=2}|\mathcal{N}_l(\rho, k)|^{2(1-1/k)},
\end{split}
 \end{align}
   where we define
\begin{align*}
\begin{split}
 {\mathcal M}'_{1,j}(\gamma)= \sum_{ 2^{m+1}< p \leq T^{\alpha_1} }
 \frac{p^{-i\gamma}}{\sqrt{p}}c_j(p,k),
\end{split}
\end{align*}
  and set for any $1 \leq j \leq \mathcal T$ and $n \in \intz$,
\begin{align}
\label{cpdef}
\begin{split}
  c_j(p,n)= \frac{1}{p^{1/(\alpha_j\log T)}}\frac{\log (T^{\alpha_j}/p)}{\log T^{\alpha_j}}+n-1.
\end{split}
 \end{align}

    We note that as $0 \leq m \leq (3/\log 2)\log\log\log T $ and $T$ is large, we have
\begin{align*}
\begin{split}
 & \Big | \sum_{ p \leq 2^{m+1}  } \frac{p^{-i\gamma}}{\sqrt{p}}c_j(p,k)\Big |
\leq  (2-k)\sum_{ p < 2^{m+1}  }
 \frac{1}{\sqrt{p}} \leq \frac{100(2-k)(\log \log T )^{3/2}}{\log \log \log T }.
\end{split}
 \end{align*}

  It follows that for $\gamma \in \mathcal{S}(j)$ and large $T$,
\begin{align*}
\begin{split}
  |{\mathcal M}'_{1,j}(\gamma)| \leq & 100(2-k)(\log \log T )^{3/2}(\log \log \log T )^{-1}+|{\mathcal M}_{1,j}(\gamma)|+ \left( 1-\frac 1k \right)|k{\mathcal P}_{1}(\rho)| \leq 2.01\alpha^{-3/4}_1 \\
=& 2.01(\log \log T )^{3/2}.
\end{split}
 \end{align*}

   We apply \eqref{ezrelation} with $z= {\mathcal M}'_{1,j}(\gamma)$, $d=[e^2\alpha^{-3/4}_1]$ and argue as before to see that
\begin{align}
\label{eM1bound}
\begin{split}
\exp\Big ( 2\Re{\mathcal M}'_{1,j}(\gamma)\Big ) \ll \Big | E_{e^2\alpha^{-3/4}_1}( {\mathcal M}'_{1,j}(\gamma)) \Big |^2.
\end{split}
 \end{align}

    As we also have $|{\mathcal M}_{l, j}| \leq  \alpha^{-3/4}_l$ when $\gamma \in \mathcal{S}(j)$, we repeat our arguments to deduce, if
\[ \left| k{\mathcal P}_{l} \left( \frac 12+i\gamma \right) \right| \leq \frac{\alpha^{-3/4}_l}{1-1/k}, \]
then
\begin{align}
\label{eMNprodest1}
\begin{split}
\exp \Big ( 2 \Re {\mathcal M}_{l,j}(\gamma)\Big )|\mathcal{N}_l(\rho, k)|^{2(1-1/k)} \leq
\left( 1+O \left( e^{-\alpha^{-3/4}_l} \right) \right)\Big | E_{e^2\alpha^{-3/4}_l}( {\mathcal M}'_{l,j}(\gamma)) \Big |^2,
\end{split}
 \end{align}
   where we define, as before, for $2 \leq l \leq j$,
\begin{align*}
\begin{split}
 {\mathcal M}'_{l,j}(\gamma)= \sum_{ p \in I_l  }\frac{p^{-i\gamma}}{\sqrt{p}}c_j(p,k).
\end{split}
 \end{align*}

  On the other hand, when $|k{\mathcal P}_{l}(\frac 12+i\gamma)| > \alpha^{-3/4}_l/(1-1/k)$, we have
\begin{align}
\label{Nlbound}
\begin{split}
  |{\mathcal N}_l(\rho, k)| \le & \sum_{r=0}^{\ell_l} \frac{|k{\mathcal P}_l(\rho)|^r}{r!}  \le
|k{\mathcal P}_l(\rho)|^{\ell_l} \sum_{r=0}^{\ell_l} \left( \left(1-\frac 1k \right)\alpha^{3/4}_l \right)^{\ell_l-r} \frac{1}{r!}  \le
|(k-1)\alpha^{3/4}_l{\mathcal P}_l(\rho)|^{\ell_l}e^{\alpha^{-3/4}_l/(1-1/k)} \\
\leq &  |(k-1)e^{(e(1-1/k))^{-1}}\alpha^{3/4}_l{\mathcal P}_l(\rho)|^{\ell_l}.
\end{split}
\end{align}

  Note that we also have
\begin{align*}
\begin{split}
\exp \Big ( 2 \Re {\mathcal M}_{l,j}(\gamma)\Big ) \leq  \exp (2\alpha^{-3/4}_l) \leq 2^{\ell_l}.
\end{split}
 \end{align*}

   We now define for all $1 \leq l \leq \mathcal J$,
\begin{align*}
\begin{split}
{\mathcal Q}_l(\rho, k):=
\begin{cases}
  \Big ( 2(1-k)e^{(e(1-1/k))^{-1}}\alpha^{3/4}_l{\mathcal P}_l(\rho) \Big )^{\lceil 1-1/k \rceil\ell_l}, \quad l \neq j+1, \\
  \Big (4(1-k)\alpha^{3/4}_l{\mathcal P}_l(\rho) \Big )^{\lceil 1/(10\alpha_{l})\rceil}, \quad l=j+1.
\end{cases}
\end{split}
 \end{align*}

 It follows from the above discussion that if $|k{\mathcal P}_{l}(\frac 12+i\gamma)| > \alpha^{-3/4}_l/(1-1/k)$ and $2 \leq l \leq j$, then
\begin{align}
\label{eMNprodest2}
\begin{split}
\exp \Big ( 2 \Re {\mathcal M}_{l,j}(\gamma)\Big )\Big|\mathcal{N}_l(\rho, k)\Big|^{2(1-1/k)} \leq \Big |{\mathcal Q}_l(\rho, k) \Big|^{2}.
\end{split}
 \end{align}

    Our arguments above also allow us to deduce that for $j+2 \leq l \leq \mathcal J$,
\begin{align}
\label{Nlbound1}
\begin{split}
 & |{\mathcal N}_l(\rho, k)|^{2(1-1/k)}
\le  (1+O(e^{-\alpha^{-3/4}_l})) \Big |{\mathcal N}_l(\rho, k-1)\Big |^2+\Big |{\mathcal Q}_l(\rho, k) \Big|^{2}.
\end{split}
\end{align}

   We apply the bounds given in \eqref{eM1bound}, \eqref{eMNprodest1}, \eqref{eMNprodest2}  and \eqref{Nlbound1} in \eqref{LboundinSP0} to arrive at
\begin{align*}
\begin{split}
   \sum_{\gamma \in \mathcal{S}(j)\bigcap \mathcal{P}(m)\bigcap \mathcal{T}} & |\zeta'(\rho)|^{2} |\mathcal{N}(\rho, k)|^{2(1-1/k)} \\
\ll &  (\log T)^2e^{(2-k)2^{m/2+4}} \exp \left(\frac {2}{\alpha_j} \right) \sum_{\gamma \in \mathcal{S}(j)} \Big (2^{m/10}|P_m(\gamma)| \Big )^{2\lceil 2^{m/2}\rceil }\Big | E_{e^2\alpha^{-3/4}_1}( {\mathcal M}'_{1,j}(\gamma)) \Big |^2 \\
& \times
\prod^{j}_{l=2} \Big (\big(1+O(e^{-\alpha^{-3/4}_l})\big )\Big | E_{e^2\alpha^{-3/4}_l}( {\mathcal M}'_{l,j}(\gamma)) \Big |^2+ \Big |{\mathcal Q}_l(\rho,k) \Big|^{2} \Big ) \\
& \times  |{\mathcal N}_{j+1}(\rho, k)|^{2(1-1/k)}\prod^{\mathcal J}_{l=j+2} \Big ( (1+O(e^{-\alpha^{-3/4}_l}))|{\mathcal N}_l(\rho, k-1)|^2+\Big |{\mathcal Q}_l(\rho, k) \Big|^{2} \Big ) \\
\ll &  (\log T)^2e^{(2-k)2^{m/2+4}} \exp \left(\frac {2}{\alpha_j} \right) \sum_{\gamma \in \mathcal{S}(j)} \Big (2^{m/10}|P_m(\gamma)| \Big )^{2\lceil 2^{m/2}\rceil }\Big | E_{e^2\alpha^{-3/4}_1}( {\mathcal M}'_{1,j}(\gamma)) \Big |^2 \\
& \times
\prod^{j}_{l=2} \Big (\Big | E_{e^2\alpha^{-3/4}_l}( {\mathcal M}'_{l,j}(\gamma)) \Big |^2+ \Big |{\mathcal Q}_l(\rho,k) \Big|^{2} \Big )|{\mathcal N}_{j+1}(\rho, k)|^{2(1-1/k)}\prod^{\mathcal J}_{l=j+2}\Big ( |{\mathcal N}_l(\rho, k-1)|^2+\Big |{\mathcal Q}_l(\rho, k) \Big|^{2} \Big ),
\end{split}
\end{align*}
  where the last estimation above follows by virtue of the bound
\begin{align*}
\begin{split}
  &  \prod^{\mathcal J}_{\substack{l=2 \\ l \neq j+1}} \big(1+O(e^{-\alpha^{-3/4}_l})\big ) \ll 1.
\end{split}
\end{align*}

   We also deduce from the description on $\mathcal{S}(j)$ and above that when $j \geq 1$,
\begin{align}
\label{upperboundprodE0}
\begin{split}
\sum_{\gamma \in \mathcal{S}(j)\bigcap \mathcal{P}(m)\bigcap \mathcal{T}} & |\zeta'(\rho)|^{2} |\mathcal{N}(\rho, k)|^{2(1-1/k)} \\
\ll & (\log T)^2e^{(2-k)2^{m/2+4}} \exp \left(\frac {2}{\alpha_j} \right)
 \sum^{ \mathcal{J}}_{u=j+1} \sum_{\gamma \in \mathcal{S}(j)} \Big (2^{m/10}|P_m(\gamma)| \Big )^{2\lceil 2^{m/2}\rceil }
\Big | E_{e^2\alpha^{-3/4}_1}( {\mathcal M}'_{1,j}(\gamma)) \Big |^2 \\
& \hspace*{1cm} \times
\prod^{j}_{l=2} \Big (\Big | E_{e^2\alpha^{-3/4}_1}( {\mathcal M}'_{l,j}(\gamma)) \Big |^2+ \Big |{\mathcal Q}_l(\rho,k) \Big|^{2} \Big )\prod^{\mathcal J}_{l=j+2}\Big ( |{\mathcal N}_l(\rho, k-1)|^2+\Big |{\mathcal Q}_l(\rho, k) \Big|^{2} \Big ) \\
& \hspace*{1cm} \times \Big | \alpha^{3/4}_{j+1}{\mathcal M}_{j+1,u}(\gamma)\Big|^{2\lceil 1/(10\alpha_{j+1})\rceil }|{\mathcal N}_{j+1}(\rho, k)|^{2(1-1/k)}.
\end{split}
\end{align}

   We further simplify the right-hand expression above by noting that if $|k{\mathcal P}_{j+1}(\rho)| \leq \alpha^{-3/4}_{j+1}/(1-1/k)$, then similar to \eqref{N1bound},
\begin{align}
\label{Njplusonebound0}
\begin{split}
 |\mathcal{N}_{j+1}(\rho, k)|^{2(1-1/k)} \ll  \exp\Big ( 2(k-1)\Re {\mathcal P}_{j+1}(\rho)\Big ) \ll \exp\Big ( 2\alpha^{-3/4}_{j+1} \Big )
\ll 2^{2\lceil 1/(10\alpha_{j+1})\rceil }.
\end{split}
\end{align}
   While when $|k{\mathcal P}_{j+1}(\rho)| > \alpha^{-3/4}_{j+1}/(1-1/k)$, we have by \eqref{Nlbound},
\begin{align}
\label{Njplusonebound}
\begin{split}
|\mathcal{N}_{j+1}(\rho, k)|^{2(1-1/k)} \leq |{\mathcal Q}_{j+1}(\rho, k)|^{2}.
\end{split}
\end{align}

   We use \eqref{Njplusonebound0} and \eqref{Njplusonebound} in \eqref{upperboundprodE0} to deduce that
\begin{align*}
\begin{split}
 \sum_{\gamma \in \mathcal{S}(j)\bigcap \mathcal{P}(m)\bigcap \mathcal{T}} & |\zeta'(\rho)|^{2} |\mathcal{N}(\rho, k)|^{2(1-1/k)}  \\
\ll & (\log T)^2e^{(2-k)2^{m/2+4}} \exp \left(\frac {2}{\alpha_j} \right)
 \sum^{ \mathcal{J}}_{u=j+1} \sum_{\gamma \in \mathcal{S}(j)} \Big (2^{m/10}|P_m(\gamma)| \Big )^{2\lceil 2^{m/2}\rceil }
\Big | E_{e^2\alpha^{-3/4}_1}( {\mathcal M}'_{1,j}(\gamma)) \Big |^2 \\
& \hspace*{1cm}\times
\prod^{j}_{l=2} \Big (\Big | E_{e^2\alpha^{-3/4}_l}( {\mathcal M}'_{l,j}(\gamma)) \Big |^2+ \Big |{\mathcal Q}_l(\rho,k) \Big|^{2} \Big )\prod^{\mathcal J}_{l=j+2}\Big ( |{\mathcal N}_l(\rho, k-1)|^2+\Big |{\mathcal Q}_l(\rho, k) \Big|^{2} \Big )\\
& \hspace*{1cm} \times \Big ( \Big | 2\alpha^{3/4}_{j+1}{\mathcal M}_{j+1,u}(\gamma)\Big|^{2\lceil 1/(10\alpha_{j+1})\rceil } +\Big |{\mathcal Q}_{j+1}(\rho, k)\Big |^{2}\Big | \alpha^{3/4}_{j+1}{\mathcal M}_{j+1,u}(\gamma)\Big |^{2\lceil 1/(10\alpha_{j+1})\rceil }\Big ) \\
\ll & (\log T)^2e^{(2-k)2^{m/2+4}} \exp \left(\frac {2}{\alpha_j} \right)
 \sum^{ \mathcal{J}}_{u=j+1} \sum_{\gamma \in \mathcal{S}(j)} \Big (2^{m/10}|P_m(\gamma)| \Big )^{2\lceil 2^{m/2}\rceil }
\Big | E_{e^2\alpha^{-3/4}_1}( {\mathcal M}'_{1,j}(\gamma)) \Big |^2 \\
&\hspace*{1cm} \times
\prod^{j}_{l=2} \Big (\Big | E_{e^2\alpha^{-3/4}_l}( {\mathcal M}'_{l,j}(\gamma)) \Big |^2+ \Big |{\mathcal Q}_l(\rho,k) \Big|^{2} \Big )\prod^{\mathcal J}_{l=j+2}\Big ( |{\mathcal N}_l(\rho, k-1)|^2+\Big |{\mathcal Q}_l(\rho, k) \Big|^{2} \Big )\\
& \hspace*{1cm} \times \Big ( \Big | 2\alpha^{3/4}_{j+1}{\mathcal M}_{j+1,u}(\gamma)\Big |^{2\lceil 1/(10\alpha_{j+1})\rceil } +\Big |{\mathcal Q}_{j+1}(\rho, k)\Big |^{4}+ \Big | \alpha^{3/4}_{j+1}{\mathcal M}_{j+1,u}(\gamma)\Big |^{4\lceil 1/(10\alpha_{j+1})\rceil }\Big ).
\end{split}
\end{align*}

  As the treatments are similar, it suffices to estimate the following expression given by
\begin{align}
\label{Sdef}
\begin{split}
 S:=&(\log T)^2e^{(2-k)2^{m/2+4}} \exp \left(\frac {2}{\alpha_j} \right)
 \sum^{ \mathcal{J}}_{u=j+1} \sum_{\gamma \in \mathcal{S}(j)} \Big (2^{m/10}|P_m(\gamma)| \Big )^{2\lceil 2^{m/2}\rceil }
\Big | E_{e^2\alpha^{-3/4}_1}( {\mathcal M}'_{1,j}(\gamma)) \Big |^2 \\
& \times
\prod^{j}_{l=2} \Big (\Big | E_{e^2\alpha^{-3/4}_l}( {\mathcal M}'_{l,j}(\gamma)) \Big |^2+ \Big |{\mathcal Q}_l(\rho,k) \Big|^{2} \Big )\prod^{\mathcal J}_{l=j+2}\Big ( |{\mathcal N}_l(\rho, k-1)|^2+\Big |{\mathcal Q}_l(\rho, k) \Big|^{2} \Big ) \Big | 2\alpha^{3/4}_{j+1}{\mathcal M}_{j+1,u}(\gamma)\Big |^{2\lceil 1/(10\alpha_{j+1})\rceil } \\
:=&  \sum^{ \mathcal{J}}_{u=j+1}S_u.
\end{split}
\end{align}

  It remains to evaluate $S_u$ for a fixed $u$.  To that end, we define a totally multiplicative function $c_j(n,k)$ such that $c_j(p,k)$ is defined in \eqref{cpdef}. This allows us to write for each $1 \leq l \leq j$,
\begin{align*}
 E_{e^2\alpha^{-3/4}_l}( {\mathcal M}'_{l,j}(\gamma))=\sum_{n_l \leq T^{\alpha_{l}\lceil e^2\alpha^{-3/4}_l \rceil}}v_l(n_l) n_l^{-i\gamma},
\quad \mbox{where} \quad
 v_l(n_l)=
\begin{cases}
\displaystyle  \frac{c_j(n_1,k)b_{1,m}(n_1)}{g(n_1)\sqrt{n_1}}, & l=1, \\
\displaystyle  \frac{c_j(n_l,k)b_l(n_l)}{g(n_l)\sqrt{n_l}}, & l \neq 1.
\end{cases}
\end{align*}
  and where the function $b_{1,m}(n)$ is defined to be take values $0$ or $1$ and $b_{1,m}(n)=1$ if and only if $n$ is composed of at most $\ell_1$ primes, all from the interval $(2^{m+1}, T^{\alpha_1}]$. \newline

  Note that for $1 \leq l \leq j$ and integers $n_l$ satisfying $b_l(n_l)=1$ when $2 \leq l \leq j$ and $b_{1,m}(n_1)=1$ when $l=1$,
\begin{align}
\label{clbound}
  |c_j(n_l,k)| \leq |2-k|^{\ell_l}.
\end{align}

  We also write for each $j+1 \leq l \leq \mathcal J$,
\begin{align*}
\begin{split}
 {\mathcal N}_l(\rho, k-1)=\sum_{n_l \leq T^{\alpha_{l}\lceil e^2\alpha^{-3/4}_l \rceil}}v_l(n_l) n_l^{-i\gamma},
\end{split}
\end{align*}
  where by \eqref{Nj},
\begin{align}
\label{vldefN}
\begin{split}
 v_l(n_l)=  \frac{(k-1)^{\Omega(n_l)}}{g(n_l)\sqrt{n_l}}  b_l(n_l) .
\end{split}
\end{align}

   We define the functions $q_{l}(n)$ for $0 \leq l \leq \mathcal J$ such that $q_{l}(n)=0$ or $1$, and $q_{l}(n)=1$ if and only if $n$ is composed of exactly $\lceil 1-1/k \rceil\ell_l$ primes (counted with multiplicity), all from the interval $I_{l}$ when $l \neq 0, j+1$. While we define $q_{0}(n)=1$ (resp. $q_{j+1}(n)=1$) if and only if $n$ is composed of exactly $\lceil 2^{m/2}\rceil $ (resp. $\lceil 1/(10\alpha_{j+1})\rceil $ ) primes (counted with multiplicity), all from the interval $(2^m, 2^{m+1}]$ (resp. $I_{j+1}$). Further, we define for $0 \leq l \leq \mathcal J$,
\begin{align*}
\begin{split}
 \beta_l(n_l) =
\begin{cases}
 \Big (2^{m/10-1}\Big )^{\lceil 2^{m/2}\rceil }\lceil 2^{m/2}\rceil! , & l=0, \\
\Big ( 2\alpha^{3/4}_{l}\Big)^{\lceil 1/(10\alpha_{l})\rceil }\lceil 1/(10\alpha_{l})\rceil)!c_u(n_l,1), & l=j+1, \\
\Big ( 2(1-k)e^{(e(1-1/k))^{-1}}\alpha^{3/4}_l \Big )^{\lceil 1-1/k \rceil\ell_l}(\lceil 1-1/k \rceil\ell_l)!, &
1\leq l \leq \mathcal J, \ \ l \neq j+1.
\end{cases}
\end{split}
\end{align*}

  With the above notations, we write
\begin{align*}
\begin{split}
\Big (2^{m/10}P_m(\gamma) \Big )^{\lceil 2^{m/2}\rceil }= & \sum_{ \substack{ n_0 }} w_0(n_0)n_0^{-2i\gamma}, \quad
\Big ( 2\alpha^{3/4}_{j+1}{\mathcal M}_{j+1,u}(\gamma)\Big)^{\lceil 1/(10\alpha_{j+1})\rceil } = \sum_{ \substack{ n_{j+1} }} w_{j+1}(n_{j+1})n_{j+1}^{-i\gamma}, \\
{\mathcal Q}_l(\rho, k) =&  \sum_{ \substack{ n_l}} w_l(n_l)n_l^{-i\gamma}, \quad
1\leq l \leq \mathcal J, \ \ l \neq j+1,
\end{split}
\end{align*}
  where
\begin{align*}
\begin{split}
 w_0(n_0)= & \frac{q_{0}(n_{0})}{n_0}\frac{\beta_0(n_0)}{g(n_0)}, \quad w_{j+1}(n_{j+1}) = \frac{q_{j+1}(n_{j+1})}{\sqrt{n_{j+1}}}\frac{\beta_{j+1}(n_{j+1}) }{g(n_{j+1})}, \\
& w_l(n_l) =  \frac{q_{l}(n_{l})}{\sqrt{n_l}}\frac{ \beta_{l}(n_{l})  }{g(n_l)}, \quad
1\leq l \leq \mathcal J, \ \ l \neq j+1.
\end{split}
\end{align*}

  We note that by \eqref{Stirling}, we have for any integer $n_0$,
\begin{align}
\label{u0bound}
 \beta_0(n_0) \leq  \lceil 2^{m/2}\rceil\Big( \frac{2^{m/10-1}\lceil 2^{m/2}\rceil }{e} \Big)^{\lceil 2^{m/2}\rceil}
\ll e^{(\log \log T)^3},
\end{align}
  since we have $0 \leq m \leq (3/\log 2)\log\log\log T $. Also, for any integer $n_{j+1}$ such that $q_{j+1}(n_{j+1})=1$, we note that the estimation given \eqref{clbound} continues to hold for $k=1$ and $j=s$. This together with \eqref{Stirling} implies that
\begin{align}
\label{ujbound}
\begin{split}
 & \beta_{j+1}(n_{j+1})
\leq  \lceil 1/(10\alpha_{j+1})\rceil\Big( \frac{2\alpha^{3/4}_{j+1}\lceil 1/(10\alpha_{j+1})\rceil }{e} \Big)^{\lceil 1/(10\alpha_{j+1})\rceil}
\ll e^{(\log \log T)^3}.
\end{split}
\end{align}

  Moreover, for any integer $n_{l}$ such that $q_{l}(n_{l})=1$ for $1 \leq l \leq \mathcal J, l \neq j+1$, applying \eqref{Stirling}, we get
\begin{align}
\label{ulbound}
\begin{split}
  \beta_l(n_l)
\leq &  \lceil 1-1/k \rceil\ell_l\Big( \frac {2(1-k)e^{(e(1-1/k))^{-1}}\alpha^{3/4}_l  \lceil 1-1/k \rceil\ell_l}{e} \Big)^{\lceil 1-1/k \rceil\ell_l} \\
\leq & \lceil 1-1/k \rceil\ell_l \Big( 2(1-k)e^{(e(1-1/k))^{-1}+2} \lceil 1-1/k \rceil \Big)^{\ell_l}.
\end{split}
\end{align}

   It follows from this, \eqref{clbound}-\eqref{vldefN} that we can write $|{\mathcal N}_l(\rho, k-1)|^2 + |{\mathcal Q}_l(\rho, k)|^{2}$ for $j+1 \leq l \leq \mathcal J$ as a Dirichlet polynomial of the form
\begin{align*}
  \sum_{n_l,n'_l \leq T^{ \lceil 1-1/k \rceil \alpha_{l}\lceil e^2\alpha^{-3/4}_l \rceil }}\frac{a_{n_l}a_{n'_l}b_l(n_l)b_l(n'_l)}{\sqrt{n_ln'_l}}n^{-i\gamma}_l{n'}^{i\gamma}_l,
\end{align*}
   where for some constant $B(k)$, whose value depends on $k$ only,
\begin{align}
\label{anbound1}
  |a_{n_l}|, |a_{n'_l}|  \leq B(k)^{\ell_j}.
\end{align}

  We may also recast
\begin{align} \label{diripoly1}
\begin{split}
 \Big (2^{m/10}|P_m(\gamma)| \Big )^{2\lceil 2^{m/2}\rceil }, \quad
\Big | E_{e^2\alpha^{-3/4}_1}( {\mathcal M}'_{1,j}(\gamma)) \Big |^2, \quad
  \Big | 2\alpha^{3/4}_{j+1}{\mathcal M}_{j+1,s}(\gamma)\Big |^{2\lceil 1/(10\alpha_{j+1})\rceil },
\end{split}
\end{align}
  and
\begin{align} \label{diripoly2}
\begin{split}
 \Big | E_{e^2\alpha^{-3/4}_l}( {\mathcal M}'_{l,j}(\gamma)) \Big |^2+ \Big |{\mathcal Q}_l(\rho,k) \Big|^{2}, \quad 2 \leq l \leq j
\end{split}
\end{align}
  into similar Dirichlet polynomials of the form
\begin{align*}
  \sum_{n,n'}\frac{a_{n}a_{n'}}{nn'}n^{-2i\gamma}{n'}^{2i\gamma} \ \ \text{or} \ \ \sum_{n,n'}\frac{a_{n}a_{n'}}{\sqrt{nn'}}n^{-i\gamma}{n'}^{i\gamma}.
\end{align*}
   We may enlarge $B(k)$ to see by \eqref{clbound}, \eqref{u0bound}-\eqref{ulbound},
\begin{align}
\label{anbound2}
  |a_{n}|, |a_{n'}|  \leq B(k)^{\ell_1}e^{(\log \log T)^3} \ll e^{2(\log \log T)^3}.
\end{align}
 Also, the lengths of the four Dirichlet polynomials in \eqref{diripoly1} and \eqref{diripoly2} are bounded by, respectively,
\begin{align}
\label{DPlengths}
 2^{((3/\log 2)\log\log\log T +1)\lceil 2^{(3/\log 2)\log\log\log T/2}\rceil }, \quad T^{ \alpha_{1}\lceil e^2\alpha^{-3/4}_1 \rceil},
\quad T^{ \alpha_{j+1}\lceil 1/(10\alpha_{j+1})\rceil }, \quad T^{ \lceil 1-1/k \rceil \alpha_{l}\lceil e^2\alpha^{-3/4}_l \rceil}.
\end{align}

  We apply the above discussions to write $S_u$ for simplicity as
\begin{align}
\label{Sudef}
\begin{split}
 S_u=& (\log T)^2e^{(2-k)2^{m/2+4}} \exp \left(\frac {2}{\alpha_j} \right)\sum_{a, b, c,d}\frac {A_{a,b,c,d}}{\sqrt{ab}cd}a^{i\gamma}b^{-i\gamma}c^{2i\gamma}d^{-2i\gamma}.
\end{split}
\end{align}
  where, by \eqref{anbound1}, \eqref{anbound2} and \eqref{DPlengths},
\begin{align*}
  A_{a,b,c,d} \ll & e^{4{\mathcal J}(\log \log T)^3} \ll e^{4(\log \log T)^4}, \\
  a, b \leq & T^{\sum^{\mathcal{J}}_{j=1}\lceil 1-1/k \rceil \alpha_{j}\lceil e^2\alpha^{-3/4}_j \rceil+\alpha_{j+1}\lceil 1/(10\alpha_{j+1})\rceil} \leq T^{40e^210^{-M/4}+1/5}, \\
 c, d \leq & 2^{((3/\log 2)\log\log\log T +1)\lceil 2^{(3/\log 2)\log\log\log T/2}\rceil } \ll T^{\varepsilon}.
\end{align*}

 It follows from this that if we apply Lemma \ref{Lem-Landau} to evaluate
\begin{align*}
\begin{split}
 \sum_{T< \gamma \leq 2T}\sum_{a, b, c,d}\frac {A_{a,b,c,d}}{\sqrt{ab}cd}a^{i\gamma}b^{-i\gamma}c^{2i\gamma}d^{-2i\gamma},
\end{split}
\end{align*}
  the error term in \eqref{sumgamma} contributes
\begin{align*}
\begin{split}
 \ll (\log T)^2\sum_{a, b, c,d}|A_{a,b,c,d}| \ll T^{1-\varepsilon}.
\end{split}
\end{align*}

 Thus, we only need to focus on the contribution from the main terms in Lemma \ref{Lem-Landau}  in the process. As the estimations are similar, it suffices to treat the sum
\begin{align}
\label{sumoverlog}
\begin{split}
& \frac {T}{2\pi}\sum_{a, b, c,d}\Big (\frac {A_{a,b,c,d}}{\sqrt{ab}cd}\frac {\Lambda((ac^2)/(bd^2))}{\sqrt{(ac^2)/(bd^2)}}+\frac {A_{a,b,c,d}}{\sqrt{ab}cd}\frac {\Lambda((bd^2)/(ac^2))}{\sqrt{(bd^2)/(ac^2)}} \Big ).
\end{split}
\end{align}

  As $\Lambda(n)$ is supported on positive prime powers, we consider the case $\Lambda((ac^2)/(bd^2))$ or $\Lambda((bd^2)/(ac^2))$ takes the value $\log q$ for a fixed prime $q$.  Without loss of generality, we may assume that $q \in I_2$ and $j>2$. It is then easy to see that we must have $c=d$ and $a, b$ can be written in the form:
\begin{align*}
\begin{split}
 a=(n_2q^{l_1})\prod_{\substack{0 \leq l \leq \mathcal J \\ l \neq 2}}n_l, \quad b=(n_2q^{l_2})\prod_{\substack{0 \leq l \leq \mathcal J \\ l \neq 2}}n_l,
\end{split}
\end{align*}
  where $l_1$, $l_2 \geq 0$, $l_1 \neq l_2$, $(n_2,q)=1$ and the integers $n_l$ satisfy (note that there is no contribution from the terms in the expansion of $|{\mathcal Q}_2(\rho, k)|^{2}$)
\begin{align*}
\begin{split}
 & q_0(n_0)=b_{1,m}(n_1)=b_2(n_2q^{l_1})=b_2(n_2q^{l_2})=q_{j+1}(n_{j+1})=1, \\
 & b^2_l(n_l)+q^2_l(n_l)=1, \quad 2 \leq l \leq \mathcal J, \ \ l \neq j+1.
\end{split}
\end{align*}

  It follows that the contribution from a fixed $q \in I_2$ to the expression in \eqref{sumoverlog} equals
\begin{align}
\label{doublesum}
\begin{split}
 \frac {T}{2\pi}\sum_{\substack{l_1 \geq 0, l_2 \geq 0 \\ l_1 \neq l_2}} & \frac {\log q}{q^{|l_1-l_2|/2}}\Big ( \sum_{\substack{n_2 \\ (n_2, q)=1}}v_{n_2q^{l_1}}v_{n_2q^{l_2}} \Big ) \sum_{\substack{n_0}}w^2_0(n_0)\sum_{\substack{n_1}}v^2_1(n_1) \sum_{\substack{n_{j+1}}}w^2_{j+1}(n_{j+1}) \prod_{\substack{3 \leq l \leq \mathcal J \\ l \neq j+1}} \sum_{\substack{n_l}}\Big (v^2_l(n_l)+w^2_l(n_l) \Big ) \\
=& \frac {T}{2\pi}\sum_{\substack{l_1 \geq 0, l_2 \geq 0 \\ l_1 \neq l_2}} \frac {\log q}{q^{|l_1-l_2|/2+(l_1+l_2)/2}}\frac {c_j(q, k)^{l_1+l_2}}{l_1!l_2!}\Big ( \sum_{\substack{n_2 \\ (n_2, q)=1}}\frac{c^2_j(n_2,k)b_2(n_2q^{l_1})b_2(n_2q^{l_2})}{g^2(n_2)n_2} \Big ) \\
& \hspace*{2cm} \times \sum_{\substack{n_0}}w^2_0(n_0)\sum_{\substack{n_1}}v^2_1(n_1) \sum_{\substack{n_{j+1}}}w^2_{j+1}(n_{j+1}) \prod_{\substack{3 \leq l \leq \mathcal J \\ l \neq j+1}} \sum_{\substack{n_l}}\Big (v^2_l(n_l)+w^2_l(n_l) \Big ).
\end{split}
\end{align}

   As in the proof of Proposition \ref{Prop4}, we remove the restriction in $b_2(n_2q^{l_1})$ on $\Omega(n_2q^{l_1})$ and in $b_2(n_2q^{l_2})$ on $\Omega(n_2q^{l_2})$ to get that the sum on the right-hand side in \eqref{doublesum} over $n_2$  becomes
\begin{align*}
  & \sum_{\substack{n_2 \\ (n_2, q)=1}}\frac{{c^2_j(n_2,k)}}{g^2(n_2)n_2}  =
  \prod_{\substack{p \in I_2 \\ p \neq q}}\left( 1+\frac{k^2}{p} +O\left( \frac 1{p^2} \right)  \right) \ll  \exp \Big(\sum_{p \in I_2}\frac{k^2}{p}+O \Big( \sum_{p \in I_2}\frac 1{p^2} \Big) \Big).
\end{align*}

  Moreover, if $\Omega(n_2q^{l_i}) \geq \ell_2$ for $i=1,2$,  then $2^{\Omega(n_2)+l_i-\ell_2}\ge 1$.  Thus, we apply Rankin's trick and the error introduced in the relaxation of the conditions on $b_2$ is
\begin{align*}
 \ll &  \big(2^{l_1}+2^{l_2} \big ) 2^{-\ell_2} \sum_{\substack{n_2 \\ (n_2, q)=1}}\frac{2^{\Omega(n_2)}|\lambda(n_{2})|^2}{g^2(n_2)n_2} \ll
\big(2^{l_1}+2^{l_2} \big ) 2^{-\ell_2/2}\exp \Big(\sum_{p \in I_2}\frac{k^2}{p}+O \Big( \sum_{p \in I_2}\frac 1{p^2} \Big) \Big).
\end{align*}

   We then deduce that the contribution to \eqref{sumoverlog} from the case $\Lambda(ac^2/bd^2)=\Lambda(bd^2/ac^2)=\log q$ for a fixed prime $q \in I_2$ is
\begin{align*}
 \sum_{\substack{l_1 \geq 0, l_2 \geq 0 \\ l_1 \neq l_2}} & \frac {\log q}{q^{|l_1-l_2|/2+(l_1+l_2)/2}}\frac {c_j(q, k)^{l_1+l_2}}{l_1!l_2!}\Big (1+O \big ( \big (2^{l_1}+2^{l_2} \big ) 2^{-\ell_2/2} \big ) \Big )  \exp \left( \sum_{p \in I_2}\frac{k^2}{p}+O\left( \sum_{p \in I_2}\frac 1{p^2} \right) \right) \\
 \ll & \Big (1+O \big (  2^{-\ell_2/2} \big )\Big ) \left( \frac{\log q}{q} + O \left( \frac
 {\log q}{q^2} \right) \right)  \exp \Big(\sum_{p \in I_2}\frac{k^2}{p}+O \Big( \sum_{p \in I_2}\frac 1{p^2} \Big) \Big).
\end{align*}

   Similar estimations carry over to the sums over $v^2_l(n_l)$ for $3 \leq l \leq \mathcal J, l \neq j+1$ in \eqref{doublesum} so that we have
\begin{align*}
\sum_{\substack{n_{l}}}v^2_{l}(n_{l}) =&  \Big (1+O \big (  2^{-\ell_l/2} \big )\Big ) \exp \Big(\sum_{p \in I_l}\frac{k^2}{p}+O \Big( \sum_{p \in I_2}\frac 1{p^2} \Big) \Big).
\end{align*}
  To treat the other sums,  we note that
\begin{align*}
\sum_{\substack{n_0}}w^2_0(n_0) \leq &  \frac {\Big (2^{m/10-1}\Big )^{2\lceil 2^{m/2}\rceil }(\lceil 2^{m/2}\rceil!)^2}{\lceil 2^{m/2}\rceil!}\Big (\sum_{2^m < p } \frac {1}{p^2}\Big )^{\lceil 2^{m/2}\rceil }, \\
\sum_{\substack{n_{j+1}}}w^2_{j+1}(n_{j+1}) \leq &  \frac {\Big ( 2\alpha^{3/4}_{j+1}\Big)^{2\lceil 1/(10\alpha_{j+1})\rceil }\Big ( \lceil 1/(10\alpha_{j+1})\rceil)! \Big )^2}{\lceil 1/(10\alpha_{j+1})\rceil)! } \Big ( \sum_{\substack{ p\in I_{j+1} } }\frac {c^2_u(p,1)}{p} \Big )^{\lceil 1/(10\alpha_{j+1})\rceil }, \\
\sum_{\substack{n_{l}}}w^2_{l}(n_{l}) \leq &  \frac {\Big ( 2(1-k)e^{(e(1-1/k))^{-1}}\alpha^{3/4}_l \Big )^{2\lceil 1-1/k \rceil\ell_l}((\lceil 1-1/k \rceil\ell_l)!)^2}{(\lceil 1-1/k \rceil\ell_l)!} \Big ( \sum_{\substack{ p\in I_{l} } }\frac {1}{p} \Big )^{\lceil 1-1/k\rceil \ell_l}, \quad l \neq j+1.
\end{align*}
  We apply \eqref{sumpj} and \eqref{Stirling} and note that $\alpha_{j+1}=2\alpha_j$, deducing from the above that
\begin{align*}
\begin{split}
\sum_{\substack{n_0}}w^2_0(n_0) \ll & e^{-m2^{m/2}/5}\exp \Big( \sum_{2^m < p \leq 2^{m+1}}\frac{k^2}{p}+O\Big( \sum_{p \in I_{j+1}}\frac 1{p^2} \Big) \Big), \\
\sum_{\substack{n_{j+1}}}w^2_{j+1}(n_{j+1})  &  \ll  e^{-\frac {80}{\alpha_{j+1}}}\exp \Big( \sum_{p \in I_{j+1}}\frac{k^2}{p}+O\Big( \sum_{p \in I_{j+1}}\frac 1{p^2} \Big) \Big)=e^{-\frac {4}{\alpha_{j}}}\exp \Big( \sum_{p \in I_{j+1}}\frac{k^2}{p}+O \Big(\sum_{p \in I_{j+1}}\frac 1{p^2} \Big) \Big), \\
\sum_{\substack{n_{l}}}w^2_{l}(n_{l}) \ll & 2^{-\ell_l/2} \exp \Big(\sum_{p \in I_l}\frac{k^2}{p}+O\Big( \sum_{p \in I_l}\frac 1{p^2} \Big) \Big).
\end{split}
\end{align*}

The bounds above enable us to derive that \eqref{doublesum} is
\begin{align*}
\begin{split}
 \ll &  e^{-m2^{m/2}/5}e^{-\frac {4}{\alpha_{j}}}\prod^{\mathcal J}_{l=1}\big(1+O(2^{-\ell_l/2})\big) \exp \Big( \sum_{p \in \bigcup^{\mathcal J}_{l=1} I_l} \frac{k^2}{p}+O\Big( \sum_{p \in \bigcup^{\mathcal J}_{l=1} I_l}\frac 1{p^2} \Big) \Big) \Big ( \frac{\log q}{q} + O \Big( \frac
 {\log q}{q^2} \Big) \Big ) \\
\ll &  e^{-m2^{m/2}/5}e^{-\frac {4}{\alpha_{j}}}\exp \Big( \sum_{p \in \bigcup^{\mathcal J}_{l=1} I_l} \frac{k^2}{p}+O\Big( \sum_{p \in \bigcup^{\mathcal J}_{l=1} I_l}\frac 1{p^2} \Big) \Big) \Big ( \frac{\log q}{q} + O\Big(\frac
 {\log q}{q^2} \Big) \Big ).
\end{split}
\end{align*}

Summing over $q$ renders that \eqref{sumoverlog} is
\begin{align*}
\begin{split}
\ll &  T e^{-m2^{m/2}/5}e^{-\frac {4}{\alpha_{j}}}\exp \Big( \sum_{p \in \bigcup^{\mathcal J}_{l=1} I_l} \frac{k^2}{p}+O\Big( \sum_{p \in \bigcup^{\mathcal J}_{l=1} I_l}\frac 1{p^2} \Big) \Big) \sum_{q \leq T^{\alpha_{\mathcal J}}} \Big (  \frac{\log q}{q} + O \Big( \frac
 {\log q}{q^2} \Big) \Big ) \\
\ll & 10^{-M} T (\log T)^{k^2+1} e^{-m2^{m/2}/5}e^{-\frac {4}{\alpha_{j}}},
\end{split}
\end{align*}
utilizing Lemma \ref{RS}.  We then conclude from this and \eqref{Sudef} that
\begin{align*}
\begin{split}
 S_u \ll & T (\log T)^{k^2+3}e^{(2-k)2^{m/2+4}-m2^{m/2}/5}e^{-\frac {2}{\alpha_{j}}}.
\end{split}
\end{align*}

Now summing over $u$ in \eqref{Sdef} gives rise to the bound
\begin{align*}
\begin{split}
 S \ll & T (\log T)^{k^2+3}e^{(2-k)2^{m/2+4}-m2^{m/2}/5}e^{-\frac {2}{\alpha_{j}}}(\mathcal{J}-j) \ll T (\log T)^{k^2+3}e^{(2-k)2^{m/2+4}-m2^{m/2}/5}e^{-\frac {1}{\alpha_{j}}},
\end{split}
\end{align*}
  where the last estimation above follows by observation that for $1 \leq j \leq \mathcal{J}$,
\begin{align*}
\mathcal{J}-j \leq \frac{\log(1/\alpha_{j})}{\log 20} \leq \frac 1{\alpha_{j}} .
\end{align*}

Finally, the bound in \eqref{sumovermj} emerges from summing over $j$ and $m$, thus completing the proof of the proposition.

\vspace*{.5cm}

\noindent{\bf Acknowledgments.}  P. G. is supported in part by NSFC grant 11871082 and L. Z. by the Faculty Silverstar Award PS65447 at the University of New South Wales (UNSW).

\bibliography{biblio}

\begin{bibdiv}
\begin{biblist}

\bib{BGM}{article}{
      author={Bui, H.~M.},
      author={Gonek, S.~M.},
      author={Milinovich, M.~B.},
       title={A hybrid {E}uler-{H}adamard product and moments of
  {$\zeta'(\rho)$}},
        date={2015},
     journal={Forum Math.},
      volume={27},
      number={3},
       pages={1799\ndash 1828},
}

\bib{C&L}{article}{
      author={Chandee, V.},
      author={Li, X.},
       title={Lower bounds for small fractional moments of {D}irichlet
  {$L$}-functions},
        date={2013},
     journal={Int. Math. Res. Not. IMRN},
      number={19},
       pages={4349\ndash 4381},
}

\bib{Da}{book}{
      author={Davenport, H.},
       title={{M}ultiplicative {N}umber {T}heory},
     edition={Third edition},
      series={Graduate Texts in Mathematics},
   publisher={Springer-Verlag},
     address={Berlin, etc.},
        date={2000},
      volume={74},
}

\bib{Gao2021-5}{article}{
      author={Gao, P.},
       title={Sharp lower bounds for moments of $\zeta'(\rho)$},
        date={Preprint},
        note={arXiv:2106.03057},
}

\bib{Gonek}{article}{
      author={Gonek, S.~M.},
       title={Mean values of the {R}iemann zeta function and its derivatives},
        date={1984},
     journal={Invent. Math.},
      volume={75},
      number={1},
       pages={123\ndash 141},
}

\bib{Gonek1}{article}{
      author={Gonek, S.~M.},
       title={On negative moments of the {R}iemann zeta-function},
        date={1989},
     journal={Mathematika},
      volume={36},
      number={1},
       pages={71\ndash 88},
}

\bib{Gonek93}{article}{
      author={Gonek, S.~M.},
       title={An explicit formula of {L}andau and its applications to the
  theory of the zeta-function},
        date={1993},
     journal={Contemp. Math.},
      volume={143},
       pages={395\ndash 413},
}

\bib{Harper}{article}{
      author={Harper, A.~J.},
       title={Sharp conditional bounds for moments of the {R}iemann zeta
  function},
        date={Preprint},
        note={arXiv:1305.4618},
}

\bib{HLZ}{article}{
      author={Heap, W.},
      author={Li, J.},
      author={Zhao, J.},
       title={Lower bounds for discrete negative moments of the {R}iemann zeta
  function},
        date={to appear},
     journal={Algebra Number Theory},
        note={arXiv:2003.09368},
}

\bib{H&Sound}{article}{
      author={Heap, W.},
      author={Soundararajan, K.},
       title={Lower bounds for moments of zeta and {$L$}-functions revisited},
        date={2022},
     journal={Mathematika},
      volume={68},
      number={1},
       pages={1\ndash 14},
}

\bib{Hejhal}{incollection}{
      author={Hejhal, D.~A.},
       title={On the distribution of {$\log|\zeta'(\frac12+it)|$}},
        date={1989},
   booktitle={Number theory, trace formulas and discrete groups ({O}slo, 1987),
  {A}cademic {P}ress, {B}oston, {MA}},
       pages={343\ndash 370},
}

\bib{HKO}{article}{
      author={Hughes, C.~P.},
      author={Keating, J.~P.},
      author={O'Connell, N.},
       title={Random matrix theory and the derivative of the {R}iemann zeta
  function},
        date={2000},
     journal={R. Soc. Lond. Proc. Ser. A Math. Phys. Eng. Sci.},
      volume={456},
      number={2003},
       pages={2611\ndash 2627},
}

\bib{iwakow}{book}{
      author={Iwaniec, H.},
      author={Kowalski, E.},
       title={{A}nalytic {N}umber {T}heory},
      series={American Mathematical Society Colloquium Publications},
   publisher={American Mathematical Society},
     address={Providence},
        date={2004},
      volume={53},
}

\bib{Kirila}{article}{
      author={Kirila, S.},
       title={An upper bound for discrete moments of the derivative of the
  {R}iemann zeta-function},
        date={2020},
     journal={Mathematika},
      volume={66},
      number={2},
       pages={475\ndash 497},
}

\bib{Landau1912}{article}{
      author={Landau, E.},
       title={\"{U}ber die {N}ullstellen der {Z}etafunktion},
        date={1912},
        ISSN={0025-5831},
     journal={Math. Ann.},
      volume={71},
      number={4},
       pages={548\ndash 564},
}

\bib{Milinovich}{article}{
      author={Milinovich, M.~B.},
       title={Upper bounds for moments of {$\zeta'(\rho)$}},
        date={2010},
     journal={Bull. Lond. Math. Soc.},
      volume={42},
      number={1},
       pages={28\ndash 44},
}

\bib{MN}{article}{
      author={Milinovich, M.~B.},
      author={Ng, N.},
       title={A note on a conjecture of {G}onek},
        date={2012},
     journal={Funct. Approx. Comment. Math.},
      volume={46},
      number={part 2},
       pages={177\ndash 187},
}

\bib{MN1}{article}{
      author={Milinovich, M.~B.},
      author={Ng, N.},
       title={Lower bounds for moments of {$\zeta'(\rho)$}},
        date={2014},
     journal={Int. Math. Res. Not. IMRN},
      number={12},
       pages={3190\ndash 3216},
}

\bib{MVa}{article}{
      author={Montgomery, H.~L.},
      author={Vaughan, R.~C.},
       title={The large sieve},
        date={1973},
     journal={Mathematika},
      volume={20},
       pages={119\ndash 134},
}

\bib{Ng04}{article}{
      author={Ng, N.},
       title={The distribution of the summatory function of the {M}\"{o}bius
  function},
        date={2004},
     journal={Proc. London Math. Soc. (3)},
      volume={89},
      number={2},
       pages={361\ndash 389},
}

\bib{Ng}{article}{
      author={Ng, N.},
       title={The fourth moment of {$\zeta'(\rho)$}},
        date={2004},
     journal={Duke Math. J.},
      volume={125},
      number={2},
       pages={243\ndash 266},
}

\bib{Ng08}{article}{
      author={Ng, N.},
       title={Extreme values of {$\zeta'(\rho)$}},
        date={2008},
     journal={J. Lond. Math. Soc. (2)},
      volume={78},
      number={2},
       pages={273\ndash 289},
}

\bib{R&Sound}{article}{
      author={Rudnick, Z.},
      author={Soundararajan, K.},
       title={Lower bounds for moments of {$L$}-functions},
        date={2005},
     journal={Proc. Natl. Acad. Sci. USA},
      volume={102},
      number={19},
       pages={6837\ndash 6838},
}

\bib{Sound01}{article}{
      author={Soundararajan, K.},
       title={Moments of the {R}iemann zeta function},
        date={2009},
     journal={Ann. of Math. (2)},
      volume={170},
      number={2},
       pages={981\ndash 993},
}

\bib{ET}{book}{
      author={Titchmarsh, E.~C.},
       title={The {T}heory of the {R}iemann {Z}eta-{F}unction},
     edition={Second Edition},
   publisher={Clarendon Press},
     address={Oxford},
        date={1986},
}

\end{biblist}
\end{bibdiv}
\bibliographystyle{amsxport}

\vspace*{.5cm}

\noindent\begin{tabular}{p{6cm}p{6cm}p{6cm}}
School of Mathematical Sciences & School of Mathematics and Statistics \\
Beihang University & University of New South Wales \\
Beijing 100191 China & Sydney NSW 2052 Australia \\
Email: {\tt penggao@buaa.edu.cn} & Email: {\tt l.zhao@unsw.edu.au} \\
\end{tabular}

\end{document}